\documentclass[11pt,a4paper,reqno]{amsart}

\usepackage{vmargin,color}

\usepackage[english]{babel}
\usepackage[utf8]{inputenc}

\usepackage{bbm,amsmath,amsthm,epsfig,latexsym,marvosym}
\usepackage{amsfonts}
\usepackage{bigints}
\usepackage{amsmath}
\usepackage{amssymb}

\usepackage[colorlinks=true,linkcolor=blue,citecolor=red]{hyperref}

\def\A{\mathbb A}

\def\R{\mathbb R}
\def\N{\mathbb N}

\def\Cl{\mathrm{Cl}(\Om;\R^n)}

\def\cal{\mathcal}

\def\F{{\cal F}}
\def\G{{\cal G}}
\def\H{{\cal H}}
\def\M{{\cal M}}
\def\L{{\cal L}}

\def\a{\alpha}

\def\e{\varepsilon}

\def\G{\Gamma}
\def\Om{\Omega}

\def\ov{\overline}

{\left\lbrace\begin{array}{@{}l@{}}}%
{\end{array}\right.}

\def\pa{\partial}

\def\trace{{\rm tr}\,}
\def\Id{{\rm Id}\,}

\def\Div{{\rm div}\,}
\def\d{\, \mathrm{d}}

\def\dist{{\rm dist}}

\def\ca{\mathbbmss{1}}

\def\pared{\partial^{*}}

\def\00{{\bf 0}}
\def\dive{{\rm div}}

\def\k{\kappa}

\DeclareMathOperator*{\Glim}{\Gamma-\lim}

\newtheorem{theorem}{Theorem}[section]
\newtheorem{corollary}[theorem]{Corollary}
\newtheorem{proposition}[theorem]{Proposition}
\newtheorem{lemma}[theorem]{Lemma}

\newtheorem{ambient}[theorem]{Assumption}

\theoremstyle{definition}
\newtheorem{remark}[theorem]{Remark}

\numberwithin{equation}{section}
\numberwithin{figure}{section}

\title{Damage-driven fracture with  low-order potentials: asymptotic behavior, existence and applications}
\author{Marco Caroccia}
\author{Nicolas Van Goethem}

\email{mcaroccia@fc.ul.pt, nvgoethem@fc.ul.pt}

\address[]{Universidade de Lisboa, Faculdade de Ci\^{e}ncias,
Departamento de Matem\'{a}tica, CMAF+CIO, Alameda da Universidade, C6, 1749-016 Lisboa, Portugal}

\makeindex
\keywords{special function of bounded deformation, fracture, free discontinuity problem, $\Gamma$-convergence, phase-field approximation, geometric measure theory}

\begin{document}
     \maketitle
\tableofcontents
\begin{abstract}
We study the $\Gamma$-convergence of  damage to  fracture  energy functionals in the presence of low-order nonlinear potentials  that allows us  to model physical phenomena such as fluid-driven fracturing, plastic slip, and the satisfaction of kinematical constraints such as crack  non-interpenetration. Existence results are also addressed.
\end{abstract}
\section{Introduction}
In linearized elasticity,  the simplest model of damage-driven brittle fracture assumes that a scalar $0\leq v\leq1$ multiplies the elasticity tensor, that is thus weakened in the damage region. At the same time, following Griffith-Bourdin-Francfort-Marigo approach \cite{FM1,FM2,FM3},  a certain amount of energy is dissipated in the damage region, and one seeks the minimum of the total energy consisting of the sum of the elastic stored energy and the dissipation terms. Specifically, in \cite{AJVG} the following damage-dependent energy functional was considered\footnote{Here we discard on purpose the work of the external forces.}: 
\begin{equation}
  J_\e(u,v):=\displaystyle \int_{\Omega}  v \A e(u)\cdot e(u)dx +\frac{1}{\e}\displaystyle\int_{\Omega} \psi(v) dx,\label{Jeps}  
\end{equation}
with $\psi(v)=k$ in the damage region $\omega\subset\Omega$, zero elsewhere, $k$ a material-dependent damage coefficient,  $v\geq \alpha\e$ with $\alpha>0$, and where $\e$ represent the thickness of the damaged region, also related to the mesh size. Here $\mathbb A$ stands for one half the constant isotropic elasticity tensor. The numerical simulations done  in \cite{AJVG} have shown that model consistency under mesh refinement strongly depended on the ratio ${k}/{\e}$.
Indeed Eq. \eqref{Jeps} was used for numerical purposes as a phase-field approximation of the Griffith energy
\begin{align}
J_ G(u):=&	\int_{\Omega} \A e(u)\cdot e(u) \d x +k\H^{n-1}(J_u), \label{grif}
	\end{align}
yet without studying any rigorous convergence result as $\e\to0$.  In anti-plane elasticity, though, that is, with $\mathbb A$ one half the identity tensor, $e(u)$ replaced by $\nabla u$ where $u$ is the vertical component of the displacement field, it is well-known that \eqref{grif} is approximated in the sense of $\Gamma$-convergence by the Ambrosio-Tortorelli functional 
\begin{equation}
  AT_\e(u,v):=\displaystyle \int_{\Omega}  \left((v+\eta_\e)|\nabla u|^2+\frac{(1-v)^2}{\e}+\e|\nabla v|^2\right) dx,\label{AT}  
\end{equation}
where it is crucial for the residual damage to be of order $\eta_\e=o(\e)$. A general case study in function of this parameter $\eta_\e$ with $\Gamma$-convergence results in the anti-plane case was carried out in \cite{iurlano2013fracture} as based on Ambrosio-Tortorelli approximation, whereas an approximation of the type \eqref{Jeps} had been considered for the scalar case, slightly earlier  by the same authors in \cite{MasoCPAM2013}. In real elasticity, that is, for the vectorial $u$ and its symmetric gradient $e(u)$ (as well as in $n$-dimensions), the first significant $\Gamma$-convergence convergence result is found in \cite{focardi2014asymptotic}, with an Ambrosio-Tortorelli-like approximation. Recently, existence results for the original Griffith's functional have been provided in $2D$ passing by Korn-type inequalities in GSBD space \cite{FRIEDRICH201827,CFI2017} (see also \cite{CCI2017} and \cite{CCF2016}). In \cite{CC2017,CC2018} the authors manage to get rid of any artificial integrability condition on the displacement field by carefully approximating the singularities, and prove some existence results by $\Gamma$-convergence with the topology of measures.

In the present paper, with the topology of $L^1$, we are concerned with approximations as based on functionals of the type \eqref{Jeps}. Indeed, it is closer to the numerical method chosen for simulation of damage-driven fracture, in particular as far as  topological sensitivity analysis is performed, already in \cite{AJVG} and more recently in \cite{XavierEFM2017}. In particular, we stick to a simple first-order damage energy, i.e., without gradients of $v$ in the energy functional (see \cite{AVG2012} for other gradient-free approximations in other contexts). Indeed, in a recent work \cite{xavier2017topological}, a simple fracking model based on damage and fluid-driven fracture and the topological derivative concept is proposed. It consists of numerical simulations based on the minimization of an energy functional of the type
\begin{equation}
  F_\e(u,v):= J_\e(u,v) -p\int_{\Omega} \psi(v)\Div u  dx,\label{HFeps}  
\end{equation}
that models a crack filled with a fluid with an imposed hydrostatic pressure $p$ which is quasi-statically increased in order to trigger a crack opening. As a generalization of this problem, our main goal is to study the asymptotic behaviour, in $\e$, of general functionals with low-order potential of the form
\begin{equation}
  \F_\e(u,v):= J_\e(u,v) +\int_{\Omega} F(x,e(u),v) dx,\label{Feps}  
\end{equation}
where $F$ need not be positive. In particular, fracking is recovered for $F=-p\psi(v)\trace{e(u)}$, but it happens that other interesting cases can be studied as for instance (i) hydraulic fracture in porous media, (ii) plastic slip, (iii) non-interpenetration or Tresca-type conditions, just to cite  some applications that we have chosen. Our main result is the $\Gamma$-convergence of $\F_\e(u,v)$ to the limit functional
\begin{align*}
	\Phi(u):=&	\int_{\Omega} \A e(u)\cdot e(u) \d x +b \H^{n-1}(J_u)  \\
	&+a \int_{J_u} \sqrt{\A ([u](z)\odot \nu(z)) \cdot ([u](z)\odot \nu(z)) } \d \H^{n-1}(z)\\
& + \int_{\Omega} F(x,u,1)+ \int_{J_u} F_{\infty}(z,[u]\odot \nu)  \d \H^{n-1}(z),
	\end{align*}
for some appropriate coefficients $a$ and $b$ related to the choice of the damage potential $\psi$ and with $F_\infty$ denoting the recession function of the convex potential, i.e., coding the asymptotic behaviour of $F$ as $|e(u)|\to\infty$. Compactness and an original approach to existence results are also proposed in Section \ref{CER}, as well as some general results given in the Appendix. 
Let us remark that a specific such low-order potential together with a treatment of the Dirichlet boundary condition were also considered in the anti-plane case in \cite{AmbrosioElast},  with the additional condition that $F\geq0$, a restriction that we wanted to avoid in the present work.

Moreover, our aim is also to be entirely self-contained, in order for these  computations and techniques be available for the mathematical/mechanical communities in the clearest way possible. Therefore, some known results are recalled and proven in our Appendix. Precise bibliography is always provided when cross-references applies, while otherwise our arguments and proof strategy are originals. Specific references for this topic are \cite{MasoCPAM2013} and \cite{iurlano2013fracture} while general and fundamental results are found in \cite{temam1980functions,buttazzo1989semicontinuity,AT,FonsecaMuellerBLOWUP,DM,AmbrosioCorsoSNS,ACDM,BCDM,Ch1}.

\section{Notations and preliminaries}

We denote by $M^{n\times n }_{sym} $ the set of all symmetric matrices with real coefficient. Given an open bounded set $\Omega$ with Lipschitz boundary we say that a function $u\in L^1(\Omega;\R^n)$ is a function of bounded deformation if there exists a matrix-valued Radon measure $ ( (Eu)_{ij})_{i,j=1}^n$ such that for all $\varphi\in C^{\infty}_c(\Omega;\R^n),$ $\|\varphi\|_{\infty}\leq 1$ it holds
	\[
	\sum_{j=1}^n\int_{\Omega} (e(\varphi))_{ij}  u_j \d x = -\sum_{j=1}^n \int_{\Omega} \varphi_j(x)  \d (Eu)_{ij}(x) \ \ \text{for all $i=1,\ldots,n$}
	\]
where $e(\varphi):=\frac{(\nabla \varphi + \nabla^T \varphi )}{2}$ denotes the symmetric part of the gradient. Notice that, if $u_k\in BD(\Omega)$ and $u_k\rightarrow u$ in $L^1$, then $Eu_k \rightharpoonup^*Eu$. 

Analogously to the behavior of the function of Bounded Variation we can identify three distinct part of the matrix valued measure $Eu$: the absolutely continuous part, the jump part (supported on $J_u$, a $(n-1)$-rectifiable set) and a Cantor part. Namely, for a generic $u\in BD(\Omega;\R^n)$, we can write
	\[
	Eu=e(u) \L^n + [u] \odot \nu_u \H^{n-1}\llcorner J_u + E^cu
	\]
where $\nu_u(x)$ is any unitary vector field orthogonal to $J_u$, $[u] = u^+-u^-$ the jump of $u$ with $u^\pm$ the approximate limit of $u$ as we approach $J_u$ and 
	\[
	[u]\odot \nu_u:=\frac{ [u]\otimes \nu_u + \nu_u \otimes [u]}{2}.
	\]
Note that in general symbol $\odot$ stands for the symmetric sum. Finally we define the space $SBD^2(\Omega;\R^n)$ as follows:
	\[
	SBD^2(\Omega;\R^n):=\{u\in BD(\Omega;\R^n) \ | \ E^c u=0,\  e(u)\in L^2(\Omega; M^{n\times n}_{sym}), \ \H^{n-1}(J_u)<+\infty\}.
	\]

\subsection{Settings of the problem}\label{sct:settings}
We consider a fourth order tensor $\A: M^{n\times n }_{sym}   \rightarrow M^{n\times n }_{sym}  $ such that there exist a constant $\kappa$ for which
	\[
	\kappa^{-1} |M|^2\leq \A M \cdot M \leq \kappa |M|^2
	\] 
where $M\cdot L:=\text{tr}(ML^T)$ is the standard scalar product inducing the Frobenius norm which, for a generic $M\in M_{sym}^{n\times n }$, is here denoted by $|M|$. \\

Having fixed $\alpha>0$ we define
	\[
	V_{\e}:= \{v\in W^{1,\infty}(\Omega) \ | \ \e\alpha < v \leq 1, \ |\nabla v_{\e}|\leq 1/\e\}.
	\]
On this space we define the sequence of energy functionals
\begin{equation}\label{eqn:energy e}
\F_{\e}(u,v):=
\left\{
\begin{array}{ll}
\begin{array}{rl}
\displaystyle \int_{\Omega}  v \A e(u)\cdot e(u) +&\displaystyle\frac{1}{\e}\displaystyle\int_{\Omega} \psi(v) d x  \\
&+ \displaystyle\int_{\Omega}F(x,e(u),v) \d x \ \ 
\end{array}  & \ \  \text{if $(u,v) \in H^1(\Omega;\R^n) \times V_{\e}$}\\
\text{}& \\
+\infty \ \ &\ \ \text{otherwise}
\end{array}
\right.
\end{equation}
where $\psi$ is any decreasing, convex function such that $\psi(1)=0$ and $F$ is a generic potential subject to the following hypothesis. 
\begin{ambient}[On the potential $F$]\label{hyp on F}
The function $F:\R^n \times M_{sym}^{n\times n}\times [0,1]\rightarrow \R$ satisfies the following properties:
\begin{itemize}
\item[1)] $F(\cdot , M  ,0) $ is continuous for all $M\in M_{sym}^{n\times n}$;
\item[2)] $F(x,\cdot,0)$ and $F(x,\cdot ,1)$ are convex for all $x\in \Omega$;
\item[3)] $ -\sigma |M|	\leq F(x,M,v)\leq \ell |M|$, for all $(x,M,v)\in \R^n \times M_{sym}^{n\times n}\times [0,1]$ where $\ell>0$ can be any real constant and 
	\begin{equation}\label{bound on p}
	0<\sigma< \max_{\lambda\in (0,1)}\left\{ \frac{2\sqrt{\alpha \psi(\lambda)}}{\sqrt{\kappa}(1+2\sqrt{\alpha |\Omega|\psi(\lambda)/\lambda})}\right\}   < 2\sqrt{\frac{\alpha \psi(0)}{ \kappa }};
	\end{equation}

\item[4)] having set 
\begin{align*}
\omega_F(s;1)&:=\sup\left\{ \frac{|F(x,M,s)-F(x,M,1)|}{|M|} \ : \ (x,M)\in \R^n\times M_{sym}^{n\times n}\right\},\\
\omega_F(s;0)&:=\sup\left\{ \frac{|F(x,M,s)-F(x,M,0)|}{|M|} \ : \ (x,M)\in \R^n\times M_{sym}^{n\times n}\right\},
\end{align*}
then 
	\[
	\displaystyle\lim_{s\rightarrow 1} \omega_F(s;1)=\lim_{s\rightarrow 0} \omega_F(s;0)=0.
	\]
\end{itemize}
\end{ambient}
\begin{remark}
In particular, $F$ can be taken as negative as we want by simply taking $\alpha \psi(0)$ large enough. 
\end{remark}
\begin{remark}\label{rmk: Lipschitz function}
We remark  that, for any fixed $x$, since $f(M)=F(x,M,0)$ is convex and satisfies $f(M)\leq \ell |M|$, then $f$ is a Lipschitz function with constant $\ell$. Indeed, consider a convex function $f:\R^n\rightarrow \R$ (with $n>1$) such that $f(x)\leq \ell |x| $ and notice that, for any $v\in S^{n-1}$, $g(t):=f(x+tv)$ is still convex and meets the requirement $g(t)\leq \ell |x+tv|$.  In particular
$\lim_{t\rightarrow +\infty} \frac{g(t)-g(0)}{t}\leq \ell$ 
and since the map $t\mapsto \frac{g(t)-g(0)}{t}$ is increasing we get
$ \frac{g(t)-g(0)}{t}\leq \ell $
for all $t\in \R$, leading to $g'(0)\leq \ell$ and thus to
	\[
	\nabla f(x) \cdot v \leq \ell \ \ \ \text{for all $v\in S^{n-1}$} \ \ \ \ \Rightarrow \ \ \ \ |\nabla f(x) |\leq \ell.
	\]
\end{remark}
We are interested in the asymptotic behavior (as $\e\to0$) of the sequence of energies \eqref{eqn:energy e}.  In particular the first aim of this paper is to show that the family of functional $\F_{\e}$, under the assumptions in \ref{hyp on F}, is $\G$-converging to the energy
	\begin{align*}
	\Phi(u):=&	\int_{\Omega} \A e(u)\cdot e(u) \d x + \int_{\Omega} F(x,u,1)  \d x\\
	&+a \int_{J_u} \sqrt{\A ([u](z)\odot \nu(z)) \cdot ([u](z)\odot \nu(z)) } \d \H^{n-1}(z)\\
&+b \H^{n-1}(J_u) + \int_{J_u} F_{\infty}(z,[u]\odot \nu)  \d \H^{n-1}(z)
	\end{align*}
defined for $u\in SBD^2(\Omega;\R^n)$ and extended to $+\infty$ otherwise. Here we have set, for the sake of shortness,
	\begin{align*}
	 &a=2\sqrt{\alpha \psi(0)}, &b=2\int_0^1 \psi(t) \d t 
	\end{align*}
 and
	\[
	F_{\infty}(z,M):=\lim_{t\rightarrow +\infty} \frac{F(z,tM,0) - F(z,0,0)}{t} \ \ \ \ \ \ \text{for $z\in J_u$ and $M\in M^{n\times n}_{sym}$}
	\]
(see Proposition \ref{propo: justification of limit} to see why $F_{\infty}$ is well defined for potential $F$ satisfying assumptions \ref{hyp on F} ). \\
\begin{remark}\label{rmk: on alpha}
Notice that the role of the condition $\alpha>0$ is linked, at least in the present analysis, to the possibility for $F$ to be negative. The approach here proposed seems to work also if we replace the condition $v_{\e}\geq \alpha\e$ with the condition $v_{\e}\geq \eta_{\e}$ for an $\eta_{\e}$ such that $\eta_{\e}/\e\rightarrow 0$, provided $F\geq 0$. 
 \end{remark}
\subsection{Main Theorems}
Setting
	\begin{equation}\label{eqn: g limit}
	\mathcal{F}(u,v):=\left\{
	\begin{array}{ll}
	\Phi(u), & \ \begin{array}{l}
	\text{if $u\in SBD^2(\Omega)$}\\
	\text{and $v=1$ $\L^n$-a.e. in $\Omega$},
\end{array}\\
& \text{}\\
	+\infty & \ \text{otherwise}
	\end{array}
	\right.
	\end{equation}
we are able to provide the following $\Gamma$-convergence result:
\begin{theorem}\label{thm: main thm g conv}
Provided the notations and the assumptions introduced in Subsection \ref{sct:settings} we have
	\[
	\Glim_{\e\rightarrow 0} \F_{\e}= \F \ 
	\]
on the space $ H^1(\Omega;\R^n) \cap L^{\infty}(\Omega) \times V_{\e}\subset SBD^2(\Omega)\cap L^{\infty}(\Omega;\R^n) \times L^1(\Omega)$ with respect to the convergence induced by the $L^1$ topology. In particular, the following assertions hold true:
\begin{itemize}
\item[a)] For any $(u_{\e},v_{\e})\in H^1(\Omega;\R^n) \times V_{\e}$ such that $u_{\e}\rightarrow u$, $v_{\e}\rightarrow v$ in $L^1$ we have
	\[
	\liminf_{\e\rightarrow 0} \F_{\e}(u_{\e},v_{\e}) \geq \F(u,v);
	\]
\item[b)] Let $\{\e_j\}_{j\in \N}$ be a vanishing sequence. Then for any $u\in SBD^2(\Omega)\cap L^{\infty}(\Omega;\R^n)$ there exists a subsequence $\{\e_{j_k}\}_{k\in \N}\subset \{\e_j\}_{j\in \N}$ and $(u_{k},v_{k})\in H^1(\Omega;\R^n)\times V_{\e_{j_k}}$ such that
	\[
	u_{k}\rightarrow u, \ \ \ v_{k} \rightarrow 1 \ \ \ \ \text{in $ L^2$, and} \ \ \  
	\lim_{k\rightarrow +\infty} \F_{\e_{j_k}}(u_{k},v_{k}) = \F(u,1).
	\]
\end{itemize}
	\end{theorem}
Moreover, we prove that the sequences with bounded energy are compact with respect to the $L^1$ topology. Namely the following theorem holds true:
	\begin{theorem}\label{main thm: comp}
With the notations and the assumptions introduced in Subsection \ref{sct:settings}, if $(u_{\e},v_{\e})\in H^1(\Omega;\R^n) \times V_{\e}$ are sequences such that
	\[
	\sup_{\e} \{\|u_{\e}\|_{L^1}+\F_{\e} (u_{\e},v_{\e})\}<+\infty
	\]
then there exists two subsequences $\{(u_{\e_k},v_{\e_k})\}_{k\in \N} \subset \{(u_{\e},v_{\e})\}_{\e>0}$ and $u\in SBD^2(\Omega)$ such that
	\[
	u_{\e_k} \rightarrow u, \ \ \ v_{\e_k} \rightarrow 1 \ \ \ \text{in $L^1$}
	\]
 and $Eu_k\rightharpoonup^* Eu$. Moreover, for any $\lambda\in(0,1)$ it holds
	\[
	e(u_{k})\ca_{\{v_{\e_k} \geq \lambda\}}\rightharpoonup e(u) \ \ \ \text{in $L^2(\Omega;M^{n\times n}_{sym})$}.
	\]	
	\end{theorem}
The proof of Theorem \ref{thm: main thm g conv} is obtained by separately proving statement a) (in Section \ref{sct Liminf}, Theorem \ref{thm: liminf inequality}) and statement b) (in Section  \ref{sct Limsup}, Theorem \ref{thm: limsup inequality} ). The compactness Theorem is proven in Subsection \ref{sct Comp} and it is basically a consequence of Propositions \ref{prop: bnd energy AT} and \ref{prop: information on convergence} in Section \ref{sct Liminf}. For the existence of minimizers with prescribed Dirichlet boundary condition we send the reader to Subsection \ref{Exist} where, under specific additional hypothesis on the potential $F$, on the boundary data and on the domain, the relaxed problem over $\ov{\Omega}$ is treated. We finally provide some examples of applications in Section \ref{sct: Example}.

\section{Liminf inequality}\label{sct Liminf}

This section is entirely devoted to the proof of the following theorem:
\begin{theorem}\label{thm: liminf inequality}
Given $(u_{\e},v_{\e})\in H^1(\Omega;\R^n)\times V_{\e}$ such that $u_{\e}\rightarrow u$ in $L^1$ and $v_{\e}\rightarrow v$ a.e. it holds
	\[
	\liminf_{\e\rightarrow 0} \mathcal{F}(u_{\e},v_{\e}) \geq \mathcal{F}(u,v).
	\]
\end{theorem}
To achieve the proof we will analyze separately what happens on the energy restricted on the sequence of sets $\Omega_{\e}^{\lambda}=\{v_{\e}\geq \lambda\}$ and $\Omega\setminus\Omega_{\e}^{\lambda}$. We start by first gaining some information on the sequences with bounded energy. To do that we will exploit the hypothesis on the nonlinear potential $F$. Let us denote by
\begin{equation}
\mathcal{W}_{\e}(u,v):=
\left\{
\begin{array}{ll}
\displaystyle \int_{\Omega} \left( v \A e(u)\cdot e(u) + \frac{\psi(v)}{\e}\right)\d x \ \  & \ \  \text{if $(u,v) \in H^1(\Omega;\R^n) \times V_{\e}$}\\
\text{}& \\
+\infty \ \ &\ \ \text{otherwise}
\end{array}
\right.
\end{equation}
and let us observe that 
	\[
	\F_{\e}(u,v)=\mathcal{W}_{\e}(u,v)+\int_{\Omega} F(x,e(u),v)\d x.
	\]
We underline that any bounds of the type
	\[
	\sup_{\e>0} \{\mathcal{W}_{\e}(u_{\e},v_{\e})\}<+\infty
	\]
leads, as we will discuss below, to an information on the convergence of $u_{\e},v_{\e}$. We now show how to derive such kind of control starting from the boundedness of $\F_{\e}$.

\begin{proposition}\label{prop: bnd energy AT}
Under the hypothesis stated in Subsection \ref{sct:settings} on $\mathbb{A}, \psi$ and $p$, there exists a constant $C$ depending on $\alpha,\mathbb{A}, |\Omega|,\psi$ and $\sigma$ only such that
	\begin{equation}\label{eqn: bound on AT energy}
	\mathcal{W}_{\e}(u,v)< C (\F_{\e}(u,v)+1)
	\end{equation}
for all $(u,v)\in H^1(\Omega;\R^n) \times V_{\e}$.
\end{proposition}

\begin{proof}
The key point is the estimate
	\begin{equation}\label{eqn: bound part 0}
  \int_{\Omega}F(x,e(u),v)  \d x \geq - \sigma \int_{\Omega}  |e(u)| \d x.
	\end{equation}
Set
	\[
	\Omega^{\lambda}=\{v\leq \lambda\}
	\]
and notice that
	\begin{align*}
	 \int_{\Omega}  |e(u)| \d x= \int_{\Omega\setminus \Omega^{\lambda}}  |e(u)| \d x + \int_{\Omega^{\lambda}}  |e(u)| \d x
	\end{align*}
and that
\begin{align}
\int_{\Omega\setminus \Omega^{\lambda}}  |e(u)| \d x &\leq \sqrt{|\Omega|} \left(\int_{\Omega\setminus \Omega^{\lambda}}  |e(u)|^{2} \d x \right)^{1/2}\nonumber \\
&\leq \sqrt{\frac{|\Omega|}{\lambda}} \left(\int_{\Omega\setminus \Omega^{\lambda}} v |e(u)|^{2} \d x \right)^{1/2}\leq \sqrt{\kappa}\sqrt{ \frac{|\Omega|  }{ \lambda}} \sqrt{ \mathcal{W}_{\e}(u,v) }.\label{eqn: bound part 1}
	\end{align}
On the other hand,
	\begin{align}
\int_{\Omega^{\lambda}}  |e(u)| \d x &=\frac{ \sqrt{\kappa}}{2\sqrt{\alpha\psi(\lambda)}} \int_{\Omega^{\lambda}} 2\sqrt{\kappa^{-1} \psi(\lambda)} \sqrt{\frac{\alpha \e}{\e}} |e(u)| \d x \nonumber \\&
\leq\frac{ \sqrt{\kappa}}{2\sqrt{\alpha\psi(\lambda)}} \left(\int_{\Omega^{\lambda}}\alpha \e\kappa^{-1}|e(u)|^2\d x+\int_{\Omega^{\lambda}}\frac{\psi(\lambda)}{\e} \d x  \right) \nonumber\\
&\leq\frac{\sqrt{\kappa}}{2 \sqrt{\alpha\psi(\lambda)}} \left(\int_{\Omega^{\lambda}}v\mathbb{A}e(u) \cdot e(u)\d x+\int_{\Omega^{\lambda}}\frac{\psi(v)}{\e} \d x  \right) \nonumber\\
&\leq\frac{\sqrt{\kappa}}{2 \sqrt{\alpha \psi(\lambda)} } \mathcal{W}_{\e}(u,v). \label{eqn: bound part 2}
	\end{align}
In particular, by combining \eqref{eqn: bound part 0},\eqref{eqn: bound part 1} and \eqref{eqn: bound part 2} we obtain, for any $(u,v)\in H^1(\Omega;\R^n)\times V_{\e}$ 
	\begin{align}
  \int_{\Omega} F(x,e(u),v)  \d x &\geq - \sigma \sqrt{ \frac{\kappa}{\alpha}}\left[\sqrt{\frac{\alpha|\Omega| }{  \lambda} } \sqrt{\mathcal{W}_{\e}(u,v)}+\frac{1}{2\sqrt{\psi(\lambda)}} \mathcal{W}_{\e}(u,v)\right]\nonumber\\
	&\geq - \sigma \sqrt{ \frac{ \kappa}{\alpha}} (1+\mathcal{W}_{\e}(u,v) ) \left[\sqrt{\frac{\alpha |\Omega| }{\lambda} } +\frac{1}{2\sqrt{\psi(\lambda)}}  \right]\nonumber\\
	&=- \sigma (1+\mathcal{W}_{\e}(u,v) ) \left[\frac{\sqrt{\kappa} (1+2\sqrt{\alpha |\Omega| \psi(\lambda)/\lambda})}{2\sqrt{\alpha \psi(\lambda)}} \right], \label{eqn: bound on bad term}
	\end{align}
	where we have used the fact that $\sqrt{\mathcal{W}_{\e}(u,v)}$ and $\mathcal{W}_{\e}(u,v)$ are each always bounded by $(1+\mathcal{W}_{\e}(u,v) )$. Moreover, inequality \eqref{eqn: bound on bad term} holds for any $\lambda\in (0,1)$ and hence it holds for the minimum among $\lambda$ which means that
\begin{align*}
 \int_{\Omega} F(x,e(u),v) \d x &\geq - \sigma (1+\mathcal{W}_{\e}(u,v) ) \min_{\lambda\in(0,1)} \left\{\frac{\sqrt{\kappa} (1+2\sqrt{\alpha |\Omega| \psi(\lambda)/\lambda})}{2\sqrt{\alpha \psi(\lambda)}}\right\}.
	\end{align*}
Notice that Assumption 1) in \ref{hyp on F} requires that 
	\[
	\sigma < \max_{\lambda\in (0,1)}\left\{ \frac{2\sqrt{\alpha \psi(\lambda)}}{\sqrt{\kappa}(1+2\sqrt{\alpha |\Omega|\psi(\lambda)/\lambda})}\right\} = \left(\min_{\lambda\in(0,1)} \left\{\frac{\sqrt{\kappa} (1+2\sqrt{\alpha |\Omega| \psi(\lambda)/\lambda})}{2\sqrt{\alpha \psi(\lambda)}}\right\}\right)^{-1}.
	\]
In particular for some $\delta>0$ depending on $\alpha, \mathbb{A}, |\Omega|, \psi$ and $\sigma$ only  we have
\[
\sigma  \min_{\lambda\in(0,1)} \left\{\frac{\sqrt{\kappa} (1+2\sqrt{\alpha |\Omega| \psi(\lambda)/\lambda})}{2\sqrt{\alpha \psi(\lambda)}}\right\} \leq (1-\delta)
\]
leading to
	\begin{equation}\label{quella li}
		\int_{\Omega} F(x,e(u),v)\d x \geq - (1-\delta) (1+\mathcal{W}_{\e}(u,v) ) .
	\end{equation}
By exploiting \eqref{quella li} we reach
		\begin{align*}
\mathcal{F}_{\e}(u,v) &= \mathcal{W}_{\e}(u,v) +\int_{\Omega} F(x,e(u),v)\d x \\
&\geq  \mathcal{W}_{\e}(u,v) -(1-\delta)\mathcal{W}_{\e}(u,v) -(1-\delta)\geq  \delta \mathcal{W}_{\e}(u,v) -1
	\end{align*}
which, by setting $C=\delta^{-1}$, achieves the proof.
\end{proof}
Let us now analyze the behaviour of the part of the energy that lives on the set $\{v_{\e}\geq \lambda\}$. We set up some notation that will be repeatedly used in this subsection. Given a sequence $v_{\e}\in V_{\e}$ and a fixed $\lambda\in(0,1)$ we define
\[
\Omega_{\e}^{\lambda}=\{v_{\e}\leq \lambda\}.
\]
We also set
	\begin{align*}
	I_{\e}^1(\lambda)&:=\int_{\Omega\setminus \Omega_{\e}^{\lambda}} v_{\e}\A e(u_{\e})\cdot e(u_{\e}) \d x , \\
	I_{\e}^2(\lambda)&:=\int_{\Omega\setminus \Omega_{\e}^{\lambda}} \frac{\psi(v_{\e})}{\e} \d x, \\ 
	I_{\e}^3(\lambda)&:=\int_{\Omega_{\e}^{\lambda}} \left(v_{\e}\A e(u_{\e})\cdot e(u_{\e}) + \frac{\psi(v_{\e})}{\e}\right) \d x.
	\end{align*}
We also	define $\F_\e(u_{\e},v_{\e};E)$, $\F(u,v;E)$ as the functionals $\F_\e$ and $\F$ with $\Omega$ replaced by $E$. Then
	$$
	\F_\e(u_{\e},v_{\e};\Omega_{\e}^{\lambda})=	I_{\e}^3(\lambda)+\int_{\Omega_\e^\lambda}F(x,e(u_{\e}),v_{\e}) dx
	$$
	 is the part of the energy that will provide the jump terms in the limit, as Proposition \ref{propo potential} will show. Let us first treat the bulk part $\F_\e(u_{\e},v_{\e})-\F_\e(u_{\e},v_{\e};\Omega_{\e}^{\lambda})=I_{\e}^1(\lambda)+I_{\e}^2(\lambda)+\int_{\Omega\setminus\Omega_\e^\lambda}F(x,e(u_{\e}),v_{\e}) $.
	
\begin{proposition}\label{prop: information on convergence}
Let $(u_{\e},v_{\e})\in H^1(\Omega;\R^n)\times V_{\e}$ be such that 	
	\[
	u_{\e}\rightarrow u, \ \ v_{\e}\rightarrow v \ \ \ \text{in $L^1$}
	\]
and with
	\begin{equation}\label{eqn: bnd on phase field}
	\sup_{\e>0}\{\F_{\e}(u_{\e},v_{\e})\}<+\infty.
	\end{equation}
Then 
	\begin{equation}\label{eqn:uniform bnd on L1 of strain}
	\sup_{\e>0} \left\{\int_{\Omega} |e(u_{\e})|\d x\right\}<+\infty.
	\end{equation}
Moreover $u\in SBD^2(\Omega;\R^n)$, $v=1$ a.e. in $\Omega$ and  for any $\lambda>0$ it holds
\begin{equation}\label{eqn: L2 bnd}
	\sup_{\e>0} \left\{\int_{\Omega\setminus \Omega_{\e}^{\lambda}} |e(u_{\e})|^2\d x, \right\}<+\infty,
\end{equation}	
\begin{equation}\label{eqn: limit part one abs}
	\begin{split}
	\liminf_{\e\rightarrow 0} \int_{\Omega\setminus \Omega_{\e}^{\lambda }} &\left[v_{\e} \mathbb{A}e(u_{\e}) \cdot e(u_{\e}) +\frac{\psi(v_{\e})}{\e} + F(x,e(u_{\e}),v_{\e}) \right] \d x\\
	& \geq \int_{\Omega} \left[\mathbb{A} e(u)\cdot e(u)+F(x,e(u),1)) \right]\d x + 2(h(1)-h(\lambda) ) \H^{n-1}(J_u)
	\end{split}
	\end{equation}
where $	h(t):=\int_0^t \psi(\tau)\d \tau$.
\end{proposition}
\begin{proof}
Thanks to Proposition \ref{prop: bnd energy AT}, the bound \eqref{eqn: bnd on phase field} implies
	\begin{equation}\label{eqn: bnd w}
	\sup_{\e>0}\{\mathcal{W}_{\e}(u_{\e},v_{\e})\}<+\infty.
	\end{equation}
In particular
	\[
	\int_{\Omega} \psi(v_{\e})\d x \rightarrow 0
	\]
which implies $\psi(v)=0$ a.e. in $\Omega$ and thus $v=1$ a.e. in $\Omega$. Moreover, fix $\lambda\in(0,1)$ and notice that
	\begin{align}
	I_{\e}^1(\lambda)&=\int_{\Omega\setminus \Omega_{\e}^{\lambda}} v \A e(u)\cdot e(u)\d x \geq \lambda \kappa^{-1} \int_{\Omega\setminus \Omega_{\e}^{\lambda}} |e(u)|^2\d x\label{eqn: inside out}
	\end{align}
 and
	\begin{align}
	I_{\e}^3(\lambda)&=\int_{\Omega_{\e}^{\lambda}} \left(v \A e(u)\cdot e(u) +\frac{\psi(v_{\e})}{\e}\right)\d x \geq  \int_{\Omega_{\e}^{\lambda}} \left(\kappa^{-1} \alpha\e|e(u)|^2+\frac{\psi(v_{\e})}{\e}\right)\d x\nonumber \\
	&\geq 2\sqrt{ \alpha }\sqrt{\kappa^{-1}} \int_{\Omega_{\e}^{\lambda}} |e(u)|\sqrt{\psi(v_{\e})}\d x \geq \sqrt{\kappa^{-1} \alpha \psi(\lambda)}\int_{\Omega_{\e}^{\lambda}} |e(u)|\d x.
	  \label{eqn: inside in}
	\end{align}
Inequality \eqref{eqn: inside out}  implies \eqref{eqn: L2 bnd}, while \eqref{eqn: inside in}, provided a further application of Cauchy-Schwarz inequality in \eqref{eqn: inside out}, yields   \eqref{eqn:uniform bnd on L1 of strain}, that in tfurn establishes the  weak compactness in $BD$. Such a compactness in the weak topology of $BD$, together with $u_{\e}\rightarrow u$ in $L^1$, implies $u\in BD(\Omega;\R^n)$. The remaining part of the proof is obtained as a slight variation of the original arguments of \cite{focardi2014asymptotic} extended in such a way as to take into account the nonlinear potential part.\\
\text{}\\
\textbf{Step one:} \textit{proof that $u\in SBD^2(\Omega;\R^n)$.} We start from the fact that
 \begin{align*}
\sup_{\e>0}\{I_{\e}^1(\lambda)+I_{\e}^2(\lambda)+I_{\e}^3(\lambda)\}=\sup_{\e>0}\{\mathcal{W}_{\e}(u_{\e},v_{\e})\}<+\infty,	\end{align*}
which implies a uniform bound in $\e$ on each $I_{\e}^i$ for $i=1,2,3$. Thanks to the co-area formula and to the property of $v_{\e}\in V_{\e}$ (in particular to $|\nabla v_{\e}|<1/\e$) we obtain
	\begin{align*}
	I_{\e}^{2}(\lambda)&=\int_{\Omega\setminus \Omega_{\e}^{\lambda}} \frac{\psi(v_{\e})}{\e}\d x\geq\int_{\Omega\setminus \Omega_{\e}^{\lambda}} |\nabla v_{\e}| \psi(v_{\e}) \d x= \int_{\Omega\setminus \Omega_{\e}^{\lambda}} |\nabla h(v_{\e})| \d x\\
	&= \int_{h(\lambda)}^{h(1)} P(\{ h(v_{\e})> t\}; \Omega) \d t \geq (h(1) - h(\lambda) )P(\{ h(v_{\e})> t_{\e}\}; \Omega),
	\end{align*}
where in the last inequality we considered the mean value theorem to find $t_{\e}\in (h(\lambda),h(1))$. We now set
	\[
	\lambda_{\e}:=h^{-1}(t_{\e})\in (\lambda,1)
	\]
and  observe that
	\[
	P(\Omega\setminus \Omega_{\e}^{\lambda_{\e}};\Omega) \leq I_{\e}^2(\lambda),
	\]
yielding
\[
\sup_{\e>0} \{P(\Omega\setminus \Omega_{\e}^{\lambda_{\e}};\Omega)\}<+\infty.
\]
Consider $\overline{u}_{\e}:=u_{\e} \ca_{\Omega\setminus \Omega_{\e}^{\lambda_{\e}}} $ and notice that, since $v_{\e}\rightarrow 1$ (and thus $|\Omega\setminus \Omega_{\e}^{\lambda_{\e}}|\rightarrow |\Omega|$), we have $\overline{u}_{\e} \rightarrow u$ in $L^1$. Moreover $\overline{u}_{\e} \in SBV(\Omega;\R^n) \cap L^2(\Omega)$ since $|\overline{u}_{\e}|\leq |u_{\e}|\in L^2$ and, due to the chain rule formula \cite[Theorem 3.96]{AFP}
	\[
D\overline{u}_{\e}=\ca_{\Omega\setminus \Omega_{\e}^{\lambda_{\e}}} \nabla u_{\e} \L^n + u_{\e} \otimes \nu_{\pared (\Omega\setminus \Omega_{\e}^{\lambda_{\e}})} \H^{n-1}\llcorner \pared (\Omega\setminus \Omega_{\e}^{\lambda_{\e}}).
	\]
In particular, $\H^{n-1}(J_{\overline{u}_{\e}} \setminus \pared (\Omega\setminus \Omega_{\e}^{\lambda_{\e}}))=0$ and hence
	\[
	\sup_{\e>0}\{\H^{n-1}(J_{\overline{u}_{\e}})\}<+\infty.
	\]
From \eqref{eqn: inside out} we also get that $\overline{u}_{\e}\in SBD^2(\Omega;\R^n)\cap L^2(\Omega)$ with
	\begin{equation}\label{technical}
	\sup_{\e>0} \left\{\int_{\Omega} |e(\overline{u}_{\e})|^2 \d x + \H^{n-1}(J_{\overline{u}_{\e}})\right\}<+\infty.
	\end{equation}
This in particular gives us that $u\in SBD^2(\Omega;\R^n)$ and  
	\begin{equation}\label{techniconv}
	e(\overline{u}_{\e}) \rightharpoonup e(u) \ \ \ \text{weakly in $L^2(\Omega; M_{sim}^{n\times n})$}.
	\end{equation}
\text{}\\
\textbf{Step two:} \textit{proof of \eqref{eqn: limit part one abs}.} Remark that the sequence $\{\lambda_{\e} \}_{\e>0}$ defined above lies in the interval $(\lambda,1)$. In particular $\Omega\setminus \Omega_{\e}^{\lambda_{\e}}\subseteq \Omega\setminus \Omega_{\e}^{\lambda}$ and relation \eqref{techniconv}, due to the convexity of the map $M\mapsto \mathbb{A}M\cdot M$ and to the strong convergence of $v_{\e}$ to $1$ almost everywhere, means that (see for instance \cite[Theorem 2.3.1]{buttazzo1989semicontinuity})
	\begin{align}
	\liminf_{\e\rightarrow 0} \int_{\Omega\setminus \Omega_{\e}^{\lambda}} v_{\e} \mathbb{A}e(u_{\e})\cdot e(u_{\e})\d x 		&\geq \liminf_{\e\rightarrow 0} \int_{\Omega\setminus \Omega_{\e}^{\lambda_{\e}}} v_{\e} \mathbb{A}e(u_{\e})\cdot e(u_{\e})\d x 	\nonumber	\\
	&=\liminf_{\e\rightarrow 0} \int_{\Omega} v_{\e} \mathbb{A}e(\overline{u}_{\e})\cdot e(\overline{u}_{\e})\d x\nonumber\\
		&  \geq \int_{\Omega}   \mathbb{A}e(u)\cdot e(u)\d x\label{eqn:a}.
	\end{align}
Moreover
	\begin{align*}
	 \left|\int_{\Omega\setminus \Omega^{\lambda}_{ \e}} F(x,e(u_{\e}),v_{\e}) \d x - \int_{\Omega\setminus \Omega_{\e}^{\lambda}} F(x,e(u_{\e}),1) \d x\right| &\leq \int_{\Omega\setminus \Omega_{\e}^{\lambda}} \omega_F(v_{\e};1) |e(u_{\e})| \d x \\
	  \left| \int_{\Omega\setminus \Omega_{\e}^{\lambda}} F(x,e(u_{\e}),1) \d x -\int_{\Omega\setminus \Omega_{\e}^{\lambda_{\e}}}F(x,e(u_{\e}),1) \d x  \right| &\leq  \int_{\Omega_{\e}^{\lambda_{\e}}\setminus \Omega_{\e}^{\lambda}} \ell |e(u_{\e})| \d x,
	\end{align*}
where we exploited item 3): $|F(x,M,v)| \leq \ell |M|$ of Assumption \ref{hyp on F}. The above quantities are vanishing (by item  4) of Assumption \ref{hyp on F} on $F$, thanks to the fact that $|\Omega_{\e}^{\lambda_{\e}}\setminus \Omega_{\e}^{\lambda}|\rightarrow 0$ and thanks to \eqref{eqn: L2 bnd}) and hence this fact, together with the convexity of the map $M\rightarrow F(x,M,1)$, implies (using once again \eqref{techniconv} and the semicontinuity Theorem \cite[Theorem 2.3.1]{buttazzo1989semicontinuity})
	\begin{align}
	\liminf_{\e\rightarrow 0}  \int_{\Omega\setminus \Omega_{\e}^{\lambda}} F(x,e(u_{\e}),v_{\e})  \d x   &=\liminf_{\e\rightarrow 0}  \int_{\Omega\setminus \Omega_{\e}^{\lambda_{\e}} }F(x,e(u_{\e}),1)  \d x\nonumber \\
	&\geq  \int_{\Omega} F(x,e(u),1)   \d x\label{eqn:b}.
	\end{align}
To achieve the proof of \eqref{eqn: limit part one abs} we need only to show that
	\[
\liminf_{\e\rightarrow 0}	\int_{\Omega\setminus \Omega_{\e}^{\lambda}} \frac{\psi(v_{\e})}{\e}\d x \geq 2(h(1)-h(\lambda) )\H^{n-1}(J_u).
	\]
In particular we use the fact  that
	\begin{equation}\label{settenani}
	\liminf_{\e\rightarrow 0}  P(\{h(v_{\e}) \geq t\};\Omega) \geq 2\H^{n-1}(J_u) \ \ \ \text{for all $t\in (h(\lambda),h(1))$}
	\end{equation}
proved in \cite{focardi2014asymptotic} via a slicing argument as established also in \cite[25, Lemma 3.2.1]{focardi2002variational}. Relation \eqref{settenani} implies immediately that 
	\begin{align*}
	\int_{\Omega\setminus \Omega_{\e}^{\lambda }} \frac{\psi(v_{\e})}{\e} \d x\geq \int_{h(\lambda)}^{h(1)} P(\{h(v_{\e}) \geq t\};\Omega) \d t\geq 2(h(1)-h(\lambda) ) \H^{n-1}(J_u)
	\end{align*}
leading to
	\begin{equation}\label{eqn:c}
	\liminf_{\e\rightarrow 0} \int_{\Omega\setminus \Omega_{\e}^{\lambda }} \frac{\psi(v_{\e})}{\e} \d x\geq 2 (h(1)-h(\lambda) )\H^{n-1}(J_u).
	\end{equation}
By collecting \eqref{eqn:a}, \eqref{eqn:b} and \eqref{eqn:c} we deduce \eqref{eqn: limit part one abs}.
\end{proof}
We now provide the liminf inequality for the (asymptotically equivalent) remaining part of the energy on $\Omega\setminus \Omega_{\e}^{\lambda}$. In order to do so, we will need to apply Proposition \ref{prop:limit of convex}, stated in the Appendix, that is a well-known approach (inspired by \cite{buttazzo1985integral}) when dealing with local functionals. We will also use the blow-up technique originally designed in \cite{FonsecaMuellerBLOWUP}.

\begin{proposition}\label{propo potential}
Let $(u_{\e},v_{\e})\in H^1(\Omega;\R^n)\times V_{\e}$ be such that 	
	\[
	u_{\e}\rightarrow u, \ \ v_{\e}\rightarrow v \ \ \ \text{in $L^1$}
	\]
and with
	\begin{equation}\label{eqn: bnd on phase field2}
	\sup_{\e>0}\{\F_{\e}(u_{\e},v_{\e})\}<+\infty.
	\end{equation}
Then, for every $\lambda\in (0,1)$ it holds
	\begin{equation}\label{eqn: limit part two sing}
	\begin{split}
	\liminf_{\e\rightarrow 0}\int_{\Omega_{\e}^{\lambda}} [2 &\sqrt{\alpha \psi(0)}\sqrt{\mathbb{A} e(u_{\e})\cdot e(u_{\e})} + F(x,e(u_{\e}),0) ]\d x \\
	&\geq  \int_{J_u}[ 2 \sqrt{\alpha \psi(0)}\sqrt{\mathbb{A} [u]\odot \nu \cdot [u]\odot \nu} + F_{\infty}(z,[u]\odot \nu,0)  ] \d \H^{n-1}(z).
	\end{split}
		\end{equation}
\end{proposition}
\begin{proof}
Set $G:\R^n \times M_{sym}^{n\times n}\rightarrow \R^+$ to be
	\[
	G(x,M)=2\sqrt{\alpha} \sqrt{\psi(0)}\sqrt{\mathbb{A} M \cdot M}+F(x,M,0)
	\]
and notice, by the hypothesis on $F$, that $G(x,\cdot)$ is a positive convex functions on $M_{sym}^{n\times n}$ for any $x\in \Omega$. In particular $G$ satisfies the hypothesis of Proposition \ref{prop:limit of convex} and thus
	\begin{equation}\label{eqn:full limit}
	\liminf_{\e\rightarrow 0} \int_{B_r(z)} G(x,e(u_{\e})) \d x \geq  \int_{B_r(z)} G(x,e(u) ) \d x + \int_{J_u\cap B_r(z)} G_{\infty}(z,[u]\odot \nu) \d \H^{n-1}(z)
	\end{equation}
for any $B_r(z)\subset \Omega$. Let $\e_k$ be the sequence achieving 
	\[
	\liminf_{\e\rightarrow 0} \int_{\Omega_{\e}^{\lambda} } G(x,e(u_{\e}))\d x=\lim_{k\rightarrow+\infty} \int_{\Omega_{\e_{k}}^{\lambda} } G(x,e(u_{\e_{k}}))\d x
	\]
and set
	\begin{align*}
	\mu_{k}(A)&:=\int_{A\cap \Omega_{\e_k}^{\lambda}} G(x,e(u_{\e_k}) )\d x,\\
	\xi_{k}(A)&:=\int_{A} G(x,e(u_{\e_k}) )\d x\\
	\end{align*}
Notice that, due to the uniform bound on the energy $\F_{\e}$ we have
	\[
	\sup_{\e}\{ \mu_{k}(\Omega)\} < +\infty,  \ \ \ \ \ 	\sup_{\e}\{ \xi_{k}(\Omega)\} < +\infty, 
	\]
and thus, up to a subsequence (not relabeled), we can find Radon measures $\mu,\xi$  such that
	\[
	\mu_{k} \rightharpoonup^* \mu , \ \ \ \xi_{k} \rightharpoonup^* \xi.
	\]
\text{}\\
\textbf{Step one:} 
We assert that the proof of \eqref{eqn: limit part two sing} follows easily from the following fact:
	\begin{equation}\label{key fact}
	\lim_{r\rightarrow 0} \frac{\mu(B_r(z))}{r^{n-1}} =\lim_{r\rightarrow 0} \frac{\xi(B_r(z))}{r^{n-1}}.
	\end{equation}
Indeed, by assuming the validity of \eqref{key fact} we conclude that, for $\L^1-$a.e. $r>0$ it holds (because of \eqref{eqn:full limit})
	\begin{align*}
	\frac{\xi(B_r(z)}{r^{n-1}}=\frac{1}{r^{n-1}}\lim_{k\rightarrow +\infty} \frac{\xi_{k}(B_r(z)) }{r^{n-1}} \geq \frac{1}{r^{n-1}} \int_{J_u \cap B_r(z)} G_{\infty}(x,[u]\odot \nu) \d \H^{n-1}(z),
	\end{align*}
implying
	\[
	\lim_{r\rightarrow 0}\frac{\mu(B_r(z))}{r^{n-1}}=\lim_{r\rightarrow 0}\frac{\xi(B_r(z))}{r^{n-1}} \geq G_{\infty}(z,[u](z)\odot \nu(z))
	\]
for $\H^{n-1}$-a.e. $z\in J_u$. This gives
	\[
	\mu(A)\geq \int_{J_u\cap A} G_{\infty}(z,[u]\odot \nu) \d \H^{n-1}(z)
	\]
and since 
	\[
	G_{\infty}(x,M)=2\sqrt{\alpha \psi(0)}\sqrt{\mathbb{A} M \cdot M} + F_{\infty}(z,M)
	\]
we obtain \eqref{eqn: limit part two sing}.\\
\text{}\\
\textbf{Step two:} 
Let us focus on \eqref{key fact}. It suffices to check that 
$$
\lim_{r\rightarrow 0} \liminf_{k\rightarrow +\infty} \frac{1}{r^{n-1}}\int_{B_r(z)\cap( \Omega\setminus\Omega_{\e_k}^{\lambda})} G(x,e(u_{\e_k}) )\d x=0
.$$
Set $\tau=\max\{\sigma, \ell\}$. Then, clearly,
	\begin{align*}
	\left| \int_{(\Omega \setminus \Omega_{\e_k}^{\lambda})\cap B_r(z) } G(x,e(u_{\e_k}) \d x \right| &\leq \int_{(\Omega \setminus \Omega_{\e_k}^{\lambda})\cap B_r(z)} (2\sqrt{\alpha \psi(0)\kappa } +\tau) |e(u_{\e_k})|\d x \\
	&\leq (2\sqrt{\alpha \psi(0)\kappa } +\tau) \int_{(\Omega \setminus \Omega_{\e_k}^{\lambda})\cap B_r(z)} |e(u_{\e_k})|\d x \\
	&\leq C|B_r(z)|^{1/2} \left(\int_{\Omega \setminus \Omega_{\e_k}^{\lambda}\cap B_r(z)} |e(u_{\e_k})|^2\d x\right)^{1/2}.
	\end{align*}
Thus
\begin{align*}
	\frac{1}{r^{n-1}}\left| \int_{(\Omega \setminus \Omega_{\e_k}^{\lambda})\cap B_r(z) } G(x,e(u_{\e_k}) \d x \right|&\leq C r^{1/2} \left(\frac{1}{r^{n-1}}\int_{(\Omega\setminus\Omega_{\e}^{\lambda}) \cap B_r(z)} |e(u_{\e_k})|^2\d x\right)^{1/2}\\
	&\leq \frac{C r^{1/2}}{\lambda^{1/2}}  \left(\frac{1}{r^{n-1}}\int_{\Omega \cap B_r(z)} v_{\e_k} |e(u_{\e_k})|^2\d x\right)^{1/2}\\
		&\leq \frac{C r^{1/2}}{\lambda^{1/2}}  \left(\frac{1}{r^{n-1}}\int_{\Omega \cap B_r(z)}\left[ v_{\e_k} |e(u_{\e_k})|^2 + \frac{\psi(v_{\e_k})}{\e_k}\right] \d x\right)^{1/2}\\
		&=\frac{C r^{1/2}}{\lambda^{1/2}}  \left(\frac{1}{r^{n-1}}\mathcal{W}_{\e_k}(u_{\e_k},v_{\e_k};B_r(z)) \right)^{1/2}.
	\end{align*}
By virtue of Proposition \ref{prop: bnd energy AT}, if
	\[
	\lim_{r\rightarrow 0} \liminf_{k\rightarrow +\infty}  \frac{\mathcal{W}_{\e_k}(u_{\e_k},v_{\e_k};B_r(z))}{r^{n-1}} =+\infty
	\]
then
	\[
	\lim_{r\rightarrow 0} \liminf_{\e\rightarrow 0} \frac{\mathcal{F}(u_{\e},v_{\e};B_r(z))}{r^{n-1}}  =+\infty,
	\]
which means that the $(n-1)$-dimensional density of the liminf lower bound is $+\infty$ and there is nothing to prove. Conversely, it holds
	\[
	\lim_{r\rightarrow 0} \liminf_{k\rightarrow +\infty} \frac{1}{r^{n-1}}\left| \int_{(\Omega \setminus \Omega_{\e_k}^{\lambda})\cap B_r(z) } G(x,e(u_{\e_k}) \d x \right|=0,
	\]
yielding \eqref{key fact}, thence completing the proof.
\end{proof}
We are now ready to proceed to the proof of Theorem \ref{thm: liminf inequality}.

\begin{proof}[Proof of Theorem \ref{thm: liminf inequality}]
Let $(u_{\e},v_{\e})\in X\times V_{\e} $ with  $u_{\e}\rightarrow u$ and $v_{\e}\rightarrow v$ in $L^1$. \\

We can easily assume that $\sup_{\e} \{\F_{\e}(u_{\e},v_{\e})\} <+\infty$ (otherwise there is nothing to prove). Let $\lambda\in (0,1)$ to be chosen later and apply Proposition \ref{prop: information on convergence} to deduce that $v=1$ $\L^n$-a.e. in $\Omega$, $u\in SBD^2(\Omega)$ and to conclude that \eqref{eqn: limit part one abs} and \eqref{eqn: L2 bnd} are in force. Thus
\begin{equation}\label{h1}
	\begin{split}
	\liminf_{\e\rightarrow 0} \int_{\Omega\setminus \Omega_{\e}^{\lambda}}&\left[ v_{\e} \mathbb{A}e(u_{\e})\cdot e(u_{\e}) +\frac{\psi(v_{\e})}{\e}+F(x,e(u_{\e}),v_{\e}) \right] \d x\\
	& \geq \int_{\Omega} \left[\mathbb{A} e(u)\cdot e(u)+F(x,e(u),1)   \right]\d x + 2(h(1)-h(\lambda))\H^{n-1}(J_u).
	\end{split}
\end{equation}
 By writing
\begin{equation}\label{hh2}
\begin{split}
\F_{\e}(u_{\e},v_{\e}) \geq&  \int_{\Omega\setminus \Omega_{\e}^{\lambda }}\left[ v_{\e} \mathbb{A}e(u_{\e})\cdot e(u_{\e}) +\frac{\psi(v_{\e})}{\e} + F(x,e(u_{\e}),v_{\e}) \right] \d x\\
&+ \int_{\Omega_{\e}^{\lambda }} \left[ v_{\e} \mathbb{A}e(u_{\e})\cdot e(u_{\e})+\frac{\psi(v_{\e})}{\e} +F(x,e(u_{\e}),v_{\e}) \right] \d x,
\end{split}
\end{equation}
it is readily seen that it suffices to focus on the second addendum in the right-hand side of \eqref{hh2}, denoted as	$\mathcal{G}_{\e}(u_{\e},v_{\e};\lambda)$, which by Cauchy-Schwarz inequality yields
	\begin{align*}
	\mathcal{G}_{\e}(u_{\e},v_{\e})\geq \int_{\Omega_{\e}^{\lambda}} \left[2\sqrt{\alpha} \sqrt{\mathbb{A} e(u_{\e}) \cdot e(u_{\e})} \sqrt{\psi(v_{\e})} + F(x,e(u_{\e}),v_{\e}) \right] \d x.
	\end{align*}
Since $\psi(s)\rightarrow \psi(0)$ and $\omega_F(s;0) \rightarrow 0$, for some $\lambda_{\delta}$ we have that
$	|\sqrt{\psi(s)} -\sqrt{\psi(0)}| + \omega_F(s;0) \leq \delta \ \ \  \text{for all $s<\lambda_{\delta}$}$. Thus, for a suitably small $\lambda$, we have
	\begin{align*}
\left|\int_{\Omega_{\e}^{\lambda}}2\sqrt{\alpha}\sqrt{\mathbb{A} e(u_{\e}) \cdot e(u_{\e})} (\sqrt{\psi(v_{\e})} -  \sqrt{\psi(0)}) \d x	\right| & \leq 2 \delta \sqrt{\kappa \alpha}   \int_{\Omega_{\e}^{\lambda}} |e(u_{\e})|\d x
	\end{align*}
	and
	\begin{align*}
\left|\int_{\Omega_{\e}^{\lambda}} [F(x,e(u_{\e}), v_{\e}) - F(x,e(u_{\e}),0)] \d x	\right| & \leq \int_{\Omega_{\e}^{\lambda}} \omega_F(v_{\e};0) |e(u_{\e})| \d x\leq \delta \int_{\Omega_{\e}^{\lambda}} |e(u_{\e})|\d x.
	\end{align*}
In particular, according to \eqref{eqn:uniform bnd on L1 of strain}, we reach
	\begin{equation*}
	\begin{split}
	\lim_{\e\rightarrow 0}\left|\int_{\Omega_{\e}^{\lambda}}[2\sqrt{\alpha}\sqrt{\mathbb{A} e(u_{\e}) \cdot e(u_{\e})} (\sqrt{\psi(v_{\e})} -\sqrt{\psi(0)}) \d x \right|& \\
	  +\lim_{\e\rightarrow 0} \left |\int_{\Omega_{\e}^{\lambda}} F(x,e(u_{\e}), v_{\e}) - F(x,e(u_{\e}),0) ]\d x	\right|&\leq \delta C
	\end{split}
	\end{equation*}
where $C$ is a constant depending on $\lambda$ and on the sequence $u_{\e}$ only.  In particular, we have
\begin{align*}
	\liminf_{\e\rightarrow 0} \mathcal{G}_{\e}(u_{\e},v_{\e}) &\geq -\delta C+ \liminf_{\e\rightarrow 0} \int_{\Omega_{\e}^{\lambda}} \left[2\sqrt{\alpha \psi(0)}\sqrt{\mathbb{A} e(u_{\e}) \cdot e(u_{\e})}  +F(x,e(u_{\e}),0)   \right]\d x.	
	\end{align*}
By applying Proposition \ref{propo potential}, and in particular relation \eqref{eqn: limit part two sing}, we get
	\begin{align}\label{topolinogigi}
	\liminf_{\e\rightarrow 0} \mathcal{G}_{\e}(u_{\e},v_{\e}) &\geq -\delta C+ \int_{J_u} \left[2\sqrt{\alpha \psi(0)}\sqrt{\mathbb{A} [u]\odot \nu \cdot [u]\odot \nu}  +F_{\infty}(x,[u]\odot \nu,0)   \right]\d x.	
	\end{align}
Summarizing, we have shown that for any $\delta>0$ there exists a $\lambda_{\delta}$ such that, if $\lambda\leq \lambda_{\delta}$, then \eqref{topolinogigi} holds true. Moreover \eqref{h1} is in force for every $\lambda\in (0,1)$. Thus, for any $\delta>0$ it must holds
		\[
		\liminf_{\e\rightarrow 0} \F_{\e}(u_{\e},v_{\e})\geq -\delta C+ \F(u,v),
		\]	
that, by taking the limit as $\delta\rightarrow 0$, achieves the proof.
\end{proof}

\section{Limsup inequality}\label{sct Limsup}
This section is entirely devoted to the construction of a recovery sequence. We first show how to recover the energy on a special class of function $\Cl$ and then we show, with a density argument, that each function $u\in SBD^2(\Omega;\R^n)$ can be recovered. Let us define
	\begin{equation}
	\Cl:=\left\{
	\begin{array}{r}
	u\in SBV^2(\Omega;\R^n)\cap L^{\infty}(\Omega;\R^n) \cap W^{m,\infty}(\Omega\setminus \ov{J_u}; \R^n),\\
 \text{for all $m\in \N$} \\
 \text{}\\
	\text{where $\ov{J_{u}}\cap \Omega $ is the finite union $S$ of closed,} \\
	\text{pairwise disjoint $(n-1)$-dimensional simplexes}\\
	\text{ intersected with $\Omega$ and $\H^{n-1}((\ov{J_u}\cap \Omega)\setminus J_u)=0$.}
		\end{array}
		 \right\}.
	\end{equation}
\subsection{Recovery sequence in $\Cl $} Consider $u\in \Cl$ and fix once and for all a unitary vector field $\nu=\nu_u$ which is normal $\H^{n-1}$-a.e. to $K=\ov{J_u}\cap \Omega$. Notice that, since $J_{u}$ is the finite union of closed and pairwise disjoint $(n-1)$-dimensional simplexes, then the point where $\nu$ is not well-defined is a set of dimension at most $n-2$. The projection operator $P:\Omega \rightarrow K$ is well defined almost everywhere around a small tubular neighborhood $T\subset \Omega $ of $K$ and thus we can consider, for points in $T$, the \textit{signed distance}
\[
\dist(x,K)= (x-\ov{x} )\cdot \nu( \ov{x} ), \ \ \ \ov{x}=P(x).
\]
We consider a normal extension of $\nu$ on $T$. We now introduce the recovery sequence. Set $\vartheta: K \rightarrow \R$, a function such that 
	\[
	\vartheta \in W^{1,\infty}(K;\R,\H^{n-1})\cap L^{\infty}(K;\R,\H^{n-1}), \ \ \vartheta>0 \ \text{ on $J_u$},
	\]
to be chosen later. We also require that $\vartheta(x)=0$ for all $x\in K\setminus J_u$. For any $\e>0$ small enough, consider the set defined as
	\begin{align*}
	A_{\vartheta \e}&:=\{y+t\nu(y) \ | \ y\in J_u, \ \ t\in (-\vartheta(y)\e, \vartheta(y)\e)\}.
	\end{align*}
Notice that up to choose $\e$ small enough it is not restrictive to assume that $A_{\vartheta\e}$  has finitely many disconnected component well separated one from another, each of which is part of a tubular neighborhood of an $(n-1)$-dimensional hyperplane (see Figure \ref{pic:disegno}). Indeed, as explained briefly in Remark \ref{trick}, up to carefully removing the singularity of the simplex where $J_u$ lives and extending $u$ smoothly on the cut (or by arguing just in the case where $\Omega$ is a cube and the jump set is an hyperplane as it is done in \cite{focardi2014asymptotic}), we obtain (asymptotically) the same result. Note that this  machinery would only make the computations heavier without adding any relevant generality to our proof; thus we will avoid it. With the same carefulness (or by suitably modify the construction provided by Theorem \ref{thm CT}, see \cite[Remark 3]{focardi2014asymptotic}), it is not restrictive to assume also $K=\ov{J_u}\cap \Omega \subset \Omega$. 
\begin{remark}\label{trick}
\begin{figure}\label{pic:disegno}
\includegraphics[scale=0.7]{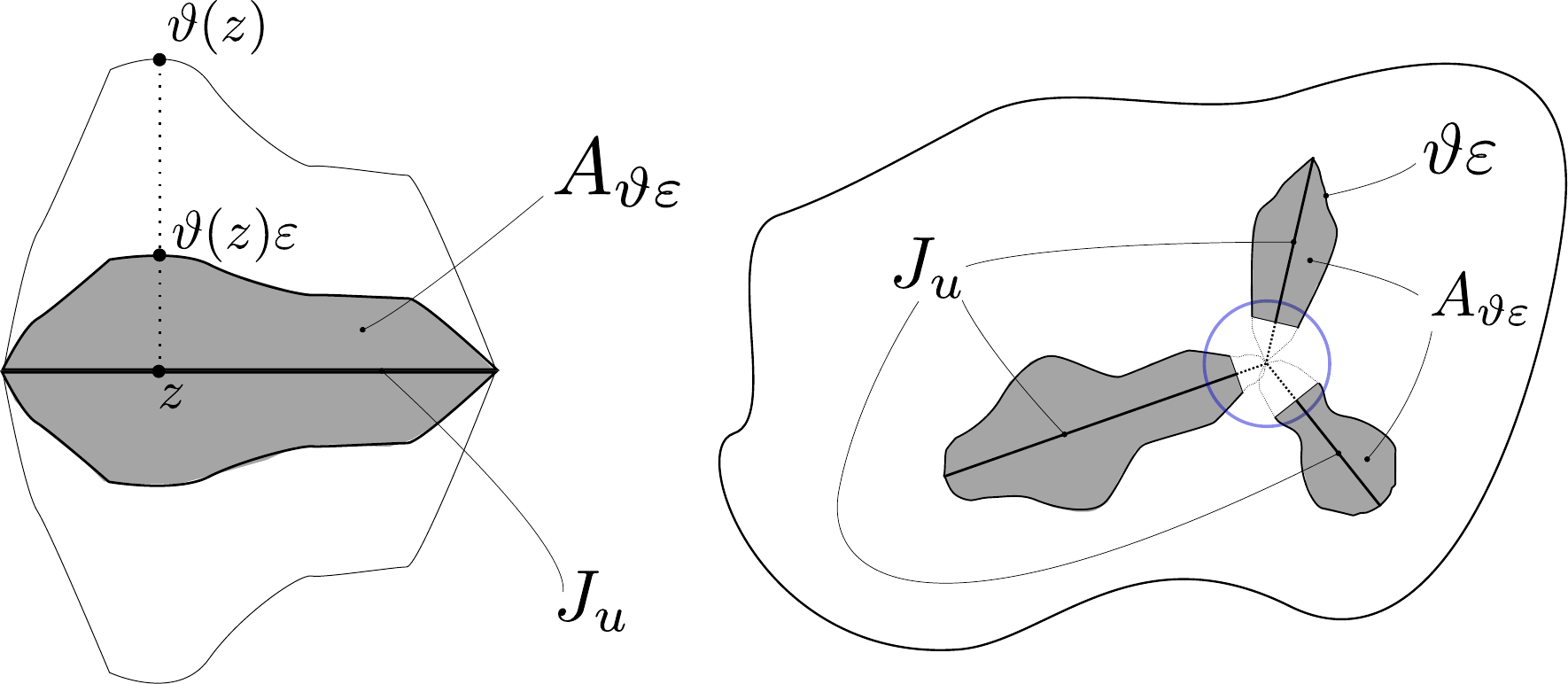}\caption{In grey is depicted the set $A_{\vartheta \e}$. In order to avoid overlaps, since the function $u_{\e}$ has been defined outside the blue ball of size $\e$, we can extend it and sew up everything together inside such a region by exploiting a capacitary argument as briefly sketched in Remarks \ref{trick} and \ref{Reg}. In particular we can always assume that the pieces of the set $A_{\vartheta \e}$, on each branches of $J_u$, do not overlap. In order to alleviate the notations we are neglecting this correction.}
\end{figure}
\rm{Let $H_1,H_2$ be two hyperplanes such that $J_{u}\subset H_1\cap H_2$. Then consider the tubular neighborhood given by the Minkowski sum $T_{1,2}(\e):= H_1\cap H_2+B_{\e}$. Assume that we are able to define our recovery sequence $u_{\e},v_{\e}$ for any $x\in \Omega \setminus T_{1,2}(\e)$. Then we can extend it to an $H^1(\Omega;\R^n)\times V_{\e}$ pair $(u_{\e},v_{\e})$ on $\Omega$ through the solution of the 2-capacity problem in $T_{1,2}(\e)$ (see \cite[proof of Corollary 3.11]{cortesani1997strong}). In particular the contribution to the energy of the pairs $(u_{\e},v_{\e})$ on the set $T_{1,2}(\e)$ is given by
	\begin{align*}
	\F_{\e}(u_{\e},v_{\e};T_{1,2}(\e))=&\int_{T_{1,2}(\e)}\left[\mathbb{A}e(u_{\e}\cdot e(u_{\e}) +F(x,e(u_{\e},v_{\e}) +\frac{\psi(v_{\e})}{\e}\right]\d x\\
	&\leq C \left[\int_{T_{1,2}(\e)}|\nabla u_{\e}|^2\d x +\frac{|T_{1,2}(\e)|}{\e}\right].
	\end{align*}
Since the $2$-capacity of $H_1\cap H_2$ in $T_{1,2}(\e)$ is $0$ and since $|T_{1,2}(\e)|/\e\rightarrow 0$ as $\e$ approaches $0$ we can conclude that the contribution to the energy of such pairs, on the set $T_{1,2}(\e)$, is asymptotically negligible. For this reason in the sequel we will assume without loss of generality that the jump set is always contained in the pairwise disjoint union of pieces of hyperplane (see Figure \ref{pic:disegno}).
}
\end{remark}
Having in mind this additional assumption on the jump set, we define the following functions
\begin{equation}\label{eqn: rcvry v regular}
v_{\e}(x)=
\left\{
\begin{array}{ll}
\displaystyle  1 & \text{if $x\notin A_{(\vartheta+1)\e} $}\\
\displaystyle  \left(\frac{1-\alpha\e}{\e }\right) |\dist (x,J_u)| -\vartheta(\ov{x}) +(1+\vartheta(\ov{x}))\alpha\e & \text{if $x \in  A_{(\vartheta+1)\e} \setminus A_{\vartheta\e} $}\\
\displaystyle \alpha\e & \text{if $ x\in A_{\vartheta \e} $}
\end{array}
\right.
\end{equation}
and
\begin{equation}\label{eqn: rcvry u regular}
u_{\e}(x)=
\left\{
\begin{array}{ll}
\displaystyle  u & \text{if $x\notin A_{\vartheta \e} $}\\
\text{}\\
\begin{array}{rr}
\displaystyle   \left(\frac{u(\ov{x}+\vartheta(\ov{x})\e \nu) - u(\ov{x}-\vartheta(\ov{x})\e \nu)  }{2\vartheta(\ov{x})\e}\right) \dist(x,Ju) &  \\
\text{}\\
 \displaystyle \ \ \ \ \ \ \ \  + \frac{u(\ov{x}+\vartheta(\ov{x})\e \nu) + u(\ov{x}-\vartheta(\ov{x})\e \nu) }{2} 
 \end{array}
 & \text{if $x \in A_{\vartheta\e}$}
\end{array}
\right..
\end{equation}
\begin{remark}[On the regularity of $(u_\e,v_\e)$]\label{Reg}
When $x$ approaches $\ov{J_u}\setminus J_u$ we have $u_{\e}(x)=u(x)$ and thus we can conclude $u_{\e}\in W^{1,\infty}(\Omega;\R^n)$. On the other hand, we see that $v_{\e}$ might present a jump on the lines 
    \[
    \{y+t\nu \ | \ y\in \ov{J_u}\setminus J_u, \ t\in (-\e, \e)\}    \]
where $\vartheta(y)=0$.  To overcome this problem we can argue as follows. 
As a consequence of \cite[Corollary 3.11, Assertion ii")]{cortesani1997strong}
we can claim that the better regularity of the jump set of $u$ ensures that 
 $\H^{n-2}(\ov{J_u}\setminus J_u)<+\infty$ and thus, for every $\e>0$ we can cover such a set with a finite number $N_{\e}$ of balls $B_{k}(\e)$ of radius $\e$ such that
$\lim_{\e \rightarrow 0} N_{\e}\e^{n-1}=0$.

Moreover, we can find a function $\zeta_{\e}$ such that $\zeta_{\e}=1$ outside $\Sigma_{\e}:=\bigcup_{k=1}^{N_{\e}} B_{k}(2\e)$, $|\nabla \zeta_{\e}|\leq 1/\e$, $\zeta_{\e}=\alpha\e$ on $\cup_k B_k(\e)$. In particular we can make use of the neighbourhood $\bigcup_{k=1}^{N_{\e}} B_{k}(3\e)\setminus \Sigma_\e$ to sew up $\zeta_{\e}\ca_{\Sigma_\e}$ with $v_\e(1-\ca_{\Sigma_\e})$ in an $H^1$ way. Furthermore, the slope of the function constructed in this way can be controlled by $1/\e$ and hence the gradient of the surgery, namely $\hat v_\e$, still has modulus less than $1/\e$ (up to the carefulness of Remark \ref{rmk costraint ve}) as required by the constraint. In particular, by considering $\hat v_\e$ in place of $v_\e$ we can see that $\hat v_{\e}\in V_\e$. In order to alleviate the notations we will neglect this correction that, indeed, 
does not affect the energy asymptotically, due to the fact that  $    |\bigcup_{k=1}^{N_{\e}} B_{k}(3\e)|/\e \leq C{N_{\e}}\e^{n-1}\rightarrow 0$.

\end{remark}
\begin{remark}[On the constraint $|\nabla v_\e|\leq 1/\e$]\label{rmk costraint ve}
Notice that
	\[
	|\nabla v_{\e}|=\frac{(1-\alpha\e)}{\e}\sqrt{[1+\e^2|\nabla \vartheta(x)|^2]}\leq C_{\e}/\e
	\]
where $C_{\e}\searrow 1$. In particular we can correct our $v_{\e}$ by dividing by the factor $C_{\e}>1$ so to ensure $|\nabla v_{\e}|\leq 1/\e$ without essentially changing the structure of the recovery sequence. To ease the notations we also decided not to take into account this small correction that is anyhow asymptotically  negligible. 
\end{remark}

Up to these modifications we can thus pretend that $u_{\e}\in W^{1,\infty}(\Omega;\R^n)$, $v_{\e}\in W^{1,\infty}(\Omega;[0,1])$ and $u_{\e} \rightarrow u$, $v_{\e}\rightarrow 1$ in $L^1$. For the sake of shortness, in the sequel when referring to a point $x\in A_{\vartheta\e}$ we will adopt the slight abuse of notation $\vartheta(x)$ by meaning $\vartheta(x)=\vartheta(\ov{x})$ which is equivalent to consider the normal extension of $\vartheta$ to $A_{\vartheta\e}$. 
We now proceed to the proof of the following Proposition.
	\begin{proposition}\label{propo: approx of nice function}
If $u\in \Cl$, there exists a function $\vartheta$ such that the sequences defined in \eqref{eqn: rcvry v regular} and \eqref{eqn: rcvry u regular} are recovery sequence for the energy $\F$. In particular
		\[
		\lim_{\e\rightarrow 0 } \F_{\e}(u_{\e},v_{\e}) =\F(u,1).
		\]
Moreover $\|u_{\e}\|_{L^{\infty}}\leq \|u\|_{L^{\infty}}$ and $u_{\e}\rightarrow u$ in $L^2$.
	\end{proposition}
\begin{proof}
We first compute the gradient of $u_{\e}$ for points $x\in A_{\vartheta\e}$.
\begin{align}
	\nabla u_{\e}(x)&=\frac{\nabla u(\ov{x}+\e\vartheta\nu)(\nabla P(x) +\e \nabla (\vartheta \nu) ) + \nabla  u(\ov{x}-\e \vartheta\nu)(\nabla P(x) -\e \nabla (\vartheta \nu) )   }{2}\nonumber\\
	&+\frac{u(\ov{x}+\vartheta\e\nu) \otimes \nu- u(\ov{x}-\vartheta\e\nu)\otimes \nu }{2\vartheta\e}\nonumber\\
	&+\frac{\nabla u(\ov{x}+\e \vartheta \nu) (\nabla P(x) +\e \nabla (\vartheta \nu)) - \nabla u(\ov{x}-\e\vartheta\nu)(\nabla P(x) - \e\nabla( \vartheta \nu))}{2\vartheta\e}\dist(x,Ju)\nonumber\\
		&-\frac{ u(\ov{x}+\e \vartheta\nu) \otimes \nabla \vartheta -  u(\ov{x}-\e\vartheta\nu)\otimes \nabla \vartheta}{2\vartheta^2\e} \dist(x,Ju)\nonumber\nonumber\\
	&=\nabla u(\ov{x}\pm\vartheta\e\nu)(\nabla P(x)\pm\e \nabla (\vartheta \nu) ) \left[\frac{1}{2}\pm\frac{\dist(x,Ju)}{2\vartheta\e}\right]\nonumber\\
		&-\frac{S_{\vartheta\e}u(\ov{x})\otimes\nabla \vartheta}{2\vartheta^2\e} \dist(x,J_u)+\frac{S_{\vartheta \e } u(\ov{x}) \otimes \nu }{2\vartheta\e},\label{gradient}
	\end{align}
where
	\[
	S_{\vartheta\e}u(\ov{x}):=u(\ov{x}+\vartheta\e\nu) - u(\ov{x}-\vartheta\e\nu).
	\]
In order to give a more clear picture of the computations we are performing, we will argue on each separate addendum of the energy $\F_{\e}$. In particular we divide the proof in three steps plus an additional fourth where we choose the appropriate $\vartheta:J_u\rightarrow \R$. Each addendum contains a principal part that has a nonzero limit as $\e$ approaches zero and a vanishing remainder $R_{\e}(u_{\e},v_{\e})$. For the sake of shortness in the sequel, we will always denote  with a small abuse, by $R_{\e}$ any term that is vanishing. In particular the term $R_{\e}$ can change from line to line. \\
	\text{}\\
	\textbf{Step one:} \textit{limit of the absolutely continuous part of the gradient}. Notice that
		\begin{align*}
		\int_{\Omega} v_{\e} \A e(u_{\e})\cdot e(u_{\e}) \d x = \int_{\Omega \setminus A_{(\vartheta+1)\e} } &\A e(u)\cdot e(u) \d x \\
		&+ \int_{A_{\vartheta\e}} \alpha\e \A e(u_{\e})\cdot e(u_{\e}) \d x + R_{\e}(u_{\e},v_{\e}),
		\end{align*}
where 
	\[
	R_{\e}(u_{\e},v_{\e}) = \int_{A_{(\vartheta+1)\e}\setminus A_{\vartheta\e}} v_{\e} \A e(u) \cdot e(u) \d x,
	\]
which (since $v_{\e}\leq  u\in W^{1,\infty}$ and $\vartheta\in L^{\infty}(K;\R, \H^{n-1})$) is clearly vanishing to $0$. Moreover
\begin{align*}
\left| e(u_{\e})(x)-\frac{S_{\vartheta\e}u(\ov{x}) \odot \nu(x)}{2\vartheta\e} \right|\leq \left[\frac{1}{2}+\frac{\dist(x,J_u)}{2\vartheta\e}\right] &\left| \nabla u (\ov{x} \pm\e\vartheta \nu)  \right| \left( \|\nabla P\|_{\infty} + \e \|\nabla (\vartheta\nu)\|_{\infty} \right)\\
&+ |S_{\vartheta\e}u(\ov{x})|\|\nabla \vartheta\|_{\infty}\frac{\dist(x,J_u)}{2\vartheta^2\e}\\
\leq \| \nabla u \|_{\infty} ( \|\nabla P\|_{\infty}&+ \e \|\nabla (\vartheta\nu)\|_{\infty} ) +\frac{ \|u\|_{\infty} \|\nabla \vartheta\|_{\infty}}{2\vartheta}\leq C ,
\end{align*}
where $C$ is a constant depending on $u$ and $\vartheta$ only (that in the sequel may vary from line to line). In particular
	\[
	\left|\A e(u_{\e})\cdot e(u_{\e}) -\frac{1}{4\vartheta^2\e^2}\A (S_{\vartheta\e}u(\ov{x}) \odot \nu(x)) \cdot S_{\vartheta\e}u(\ov{x}) \odot \nu(x)\right|\leq C\left(1+\frac{\|u\|_{\infty}}{\vartheta\e}\right).
	\]
This means that
	\begin{align*}
	\int_{A_{\vartheta\e}}\alpha\e\left|\A e(u_{\e})\cdot e(u_{\e}) -\frac{1}{4\vartheta^2\e^2}\A (S_{\vartheta\e}u(\ov{x}) \odot \nu(x)) \cdot S_{\vartheta\e}u(\ov{x}) \odot \nu(x)\right| \d x \leq \alpha\e C
	\end{align*}
implying
	\begin{align*}
	\int_{A_{\vartheta \e}} \alpha\e \A e(u_{\e})\cdot e(u_{\e}) \d x=\alpha\int_{A_{\vartheta\e}}\frac{1}{4\vartheta(\ov{x})^2\e} \A (S_{\vartheta\e}u(\ov{x}) \odot \nu(x)) \cdot S_{\vartheta\e}u(\ov{x}) \odot& \nu(x) \d x \\
	& +R_{\e}(u_{\e},v_{\e}).
	\end{align*}
By slicing the term with the co-area formula we get
	\begin{align*}
\alpha\int_{A_{\vartheta\e}}& \frac{1}{4\vartheta(\ov{x})^2\e} \A (S_{\vartheta\e}u(\ov{x}) \odot \nu(x)) \cdot S_{\vartheta\e}u(\ov{x}) \odot \nu(x) \d x\\
	&=\alpha\int_{J_u} \d \H^{n-1}(z) \int_{-\e\vartheta(z)}^{\e\vartheta(z)}\frac{1}{4\vartheta(z)^2\e} \A (S_{\vartheta\e}u(z) \odot \nu(z)) \cdot S_{\vartheta\e}u(z) \odot \nu(z) \d t	\\
		&=\alpha\int_{J_u} \frac{1}{2\vartheta(z)} \A (S_{\vartheta\e}u(z) \odot \nu(z)) \cdot S_{\vartheta\e}u(z) \odot \nu(z)  \d \H^{n-1}(z).	
	\end{align*}
By virtue of $S_{\vartheta \e}u(z) \rightarrow [u](z)$, 
we get
	\begin{equation}\label{eqn: limit part ac}
	\begin{split}
	\lim_{\e\rightarrow 0} \int_{\Omega} v_{\e} \A e(u_{\e})\cdot e(u_{\e})& \d x=\int_{\Omega} \A e(u)\cdot e(u) \d x \\
	&+\alpha\int_{J_u} \frac{\A ([u](z)\odot \nu(z)) \cdot ([u](z)\odot \nu(z)) }{2\vartheta(z)} \d \H^{n-1}(z).
	\end{split}
	\end{equation}
\text{}\\
\textbf{Step two:}\textit{ limit of the fracture's potential part.} Notice that
	\begin{align*}
	\frac{1}{\e}\int_{\Omega} \psi(v_{\e}) &= \frac{1}{\e}\int_{A_{(\vartheta+1)\e}\setminus A_{\vartheta \e} } \psi(v_{\e}) \d x +\frac{\psi(\alpha\e)}{\e}\int_{A_{\vartheta\e}} \d x\\
	&=\frac{1}{\e}\int_{J_u} \int_{\vartheta(z)\e}^{(\vartheta(z)+1)\e} \psi(v_{\e}(z+t\nu)) + \psi(v_{\e}(z-t\nu)))\d t \d \H^{n-1}(z)\\
	&  \ \ \ \ \ \ \ \ + 2\psi(\alpha\e) \int_{J_u} \vartheta(z) \d \H^{n-1}(z)\\
		&=\int_{J_u} \int_{0}^{1} \psi(v_{\e}(z+(t\e+\e\vartheta(z))\nu)) + \psi(v_{\e}(z-(t\e+\e\vartheta(z))\nu)))\d t \d \H^{n-1}(z)\\
	&  \ \ \ \ \ \ \ \ + 2\psi(\alpha\e) \int_{J_u} \vartheta(z) \d \H^{n-1}(z).\\
	\end{align*}
Since $	v_{\e}(z\pm(t\e+\e\vartheta(z))\nu) \rightarrow t$, 
we get
	\begin{equation}\label{eqn: limit part fracture}
	\begin{split}
	\lim_{\e\rightarrow 0} \frac{1}{\e}\int_{\Omega} \psi(v_{\e})\d x = 2\psi(0) \int_{J_u}\vartheta(z)\d \H^{n-1}(z) + 2\H^{n-1}(J_u)\int_{0}^1 \psi(t) \d t.
	\end{split}
	\end{equation}
	\text{}\\
\textbf{Step three:} \textit{limit of the lower order potential}.
We see that
	\begin{align*}
	\int_{\Omega} F(x,e(u_{\e}),v_{\e}) \d x &= \int_{\Omega\setminus A_{(\vartheta+1)\e}} F(x,e(u),1)\d x + \int_{ A_{(\vartheta+1)\e}\setminus A_{\vartheta\e}} F(x,e(u),v_{\e})\d x \\
	&+  \int_{  A_{\vartheta\e}} F(x,e(u_{\e}),\alpha\e)\d x \\
	&= \int_{\Omega\setminus A_{(\vartheta+1)\e}} F(x,e(u),1)\d x+  \int_{  A_{\vartheta\e}} F(x,e(u_{\e}),\alpha\e)\d x +R_{\e}(u_{\e},v_{\e}).
	\end{align*}
Once again the co-area formula leads to
	\begin{align*}
	\int_{A_{\vartheta\e}} F(x,e(u_{\e}), \alpha\e) \d x =& \int_{J_u} \int_{-\vartheta\e}^{\vartheta\e} F(z+t \nu, e(u_{\e})(z+t\nu), \alpha\e) \d t \d\H^{n-1}(z)\\
	=&\int_{J_u} \vartheta(z)\e \int_{0}^{1} F(z+t\vartheta(z)\e \nu, e(u_{\e})(z+t \vartheta(z)\e \nu), \alpha\e) \d t \d\H^{n-1}(z)\\
	&+\int_{J_u} \vartheta(z)\e \int_0^{1} F(z-t \vartheta(z)\e \nu, e(u_{\e})(z-t\vartheta(z)\e \nu), \alpha\e) \d t  \d\H^{n-1}(z).
	\end{align*}
	We underline that
\begin{align*}
	e(u_{\e})(z\pm t\e \vartheta(z) \nu) &=\frac{1}{4}\nabla u (\nabla P \pm\e \nabla (\vartheta \nu) )(z\pm\vartheta\e\nu) \left[1\pm t \right]\\
	&+ \frac{1}{4}(\nabla P^t\pm\e \nabla (\vartheta \nu)^t)\nabla u^t(z \pm\vartheta\e\nu)\left[1\pm t\right]	\\
		&-\frac{S_{\vartheta\e}u(z)\odot\nabla \vartheta}{2}\frac{t}{\vartheta(z)}+\frac{S_{\e\vartheta} u(z) \odot \nu }{2\vartheta\e}=M_{\e}+ \frac{[u](z)\odot \nu}{2\vartheta(z) \e}
\end{align*}
with
\begin{align*}
M_{\e}&:=\frac{1}{4}\nabla u (\nabla P \pm\e \nabla (\vartheta \nu) )(z\pm\vartheta\e\nu) \left[1\pm t \right]+ \frac{1}{4}(\nabla P^t\pm\e \nabla (\vartheta \nu)^t)\nabla u^t(z \pm\vartheta\e\nu)\left[1\pm t\right]	\\
		&-\frac{S_{\vartheta\e}u(z)\odot\nabla \vartheta}{2}\frac{t}{\vartheta(z)}+\frac{S_{\e\vartheta} u(z) \odot \nu  - [u]\odot \nu }{2\vartheta\e}.
\end{align*}
Note that from
	\[
	\Big{|} \frac{u(z\pm t\e\vartheta(z) \nu) - u(z)^{\pm}}{\e \vartheta(z)} \Big{|}\leq \frac{1}{2\e \vartheta(z)} \int_{0}^{t\e\vartheta(z)}|\nabla u(z\pm s \nu) \nu| \d s\leq \frac{t}{2}\|\nabla u\|_{\infty},
	\] 
clearly $\|M_{\e}\|<+\infty$. Then, by definition of $F_{\infty}$, we get
	\begin{align*}
	\lim_{\e\rightarrow 0} \vartheta(z) \e F(z+t\vartheta(z)\e \nu, e(u_{\e})&(z+t \vartheta(z)\e \nu), \alpha\e) \\
	& = \lim_{\e\rightarrow 0} \vartheta(z) \e F\left(z+t\vartheta(z)\e \nu, M_{\e} + \frac{[u]\odot \nu}{2\vartheta(z)\e}, \alpha\e\right)\\
	& =\lim_{\e\rightarrow 0} \frac{1}{2}  \frac{F\left(z+t\vartheta(z)\e \nu, M_{\e} + \frac{[u]\odot \nu}{2\vartheta(z)\e}, \alpha\e\right)}{1/(2\vartheta(z)\e)}=\frac{1}{2} F_{\infty}(z,[u]\odot \nu).
	\end{align*}
In the same token,
	\begin{align*}
	\lim_{\e\rightarrow 0} \vartheta(z) \e F(z-t\vartheta(z)\e \nu, e(u_{\e})(z-t \vartheta(z)\e \nu), \alpha\e)  &=\frac{1}{2} F_{\infty}(z,[u]\odot \nu).
	\end{align*}
Hence
	\begin{align*}
	\lim_{\e\rightarrow 0} \int_{A_{\vartheta\e}} F(x,e(u_{\e}),\alpha\e) \d x = \int_{J_u} F_{\infty}(z,[u]\odot \nu) \d \H^{n-1}(z).
	\end{align*}
In particular,
	\begin{equation}\label{eqn: limit part dive}
	\begin{split}
	\lim_{\e\rightarrow 0} \int_{\Omega} F (x,e(u_{\e}),v_{\e}) \d x = \int_{\Omega} F(x,e(u),1)\d x  +\int_{J_u} F_{\infty}(z,[u]\odot \nu)  \d \H^{n-1}(z).
	\end{split}
	\end{equation}
	\text{}\\
\textbf{Step four:} \textit{Choice of $\vartheta$}. Collecting together steps one, two and three and in particular \eqref{eqn: limit part ac}, \eqref{eqn: limit part fracture} and \eqref{eqn: limit part dive} we write
	\begin{align*}
	\lim_{\e\rightarrow 0} \F_{\e}(u_{\e},v_{\e})=&\int_{\Omega} \A e(u)\cdot e(u) \d x \\
	&+\alpha\int_{J_u} \frac{\A ([u](z)\odot \nu(z)) \cdot ([u](z)\odot \nu(z)) }{2\vartheta(z)} \d \H^{n-1}(z)\\
&+2\psi(0) \int_{J_u}\vartheta(z)\d \H^{n-1}(z) + 2\int_{0}^1 \psi(t) \d t\H^{n-1}(J_u)\\
&+\int_{\Omega} F(x,e(u),1)  \d x + \int_{J_u} F_{\infty}(z,[u]\odot \nu) \d \H^{n-1}(z).
	\end{align*}
Due to Schwarz inequality
	\[
	\frac{A}{2\vartheta} +2\vartheta B\geq 2\sqrt{AB} \ \ \ \ \text{where "$=$" is attained iff $\vartheta=\frac{\sqrt{A}}{2\sqrt{B}}$},
	\]
by choosing $\bar{\vartheta}(z):=\frac{\sqrt{\a}}{2\sqrt{\psi(0)}}\sqrt{\A ([u](z) \odot \nu(z))\cdot [u](z) \odot \nu(z)}$ we can guarantee that, for any other $\vartheta(z)$ satisfying the hypothesis, it will hold
	\begin{align*}
	\alpha\int_{J_u} \frac{\A ([u](z)\odot \nu(z)) \cdot ([u](z)\odot \nu(z)) }{2\vartheta(z)} & \d \H^{n-1}(z)
+2\psi(0) \int_{J_u}\vartheta(z)\d \H^{n-1}(z) \\
 \geq 2\sqrt{\alpha \psi(0)} \int_{J_u} & \sqrt{\A ([u](z) \odot \nu(z))\cdot [u](z) \odot \nu(z)} \d \H^{n-1}(z). 
	\end{align*}
In particular, with this choice we reach the equality (minimum energy). Notice that all the hypothesis are satisfied due to the regularity of $u\in \Cl$, in particular $\bar{\vartheta}(z) \in W^{1,\infty}(K;\R,\H^{n-1})\cap L^{\infty}(K;\R,\H^{n-1})$. Moreover, by definition it holds $\bar{\vartheta}>0$ on $J_u$ and $\ov{\vartheta}=0$ on $K\setminus J_u$. Thus, this choice guarantees that
	\begin{align*}
	\lim_{\e\rightarrow 0} \F_{\e}(u_{\e},v_{\e})=&\int_{\Omega} \A e(u)\cdot e(u) \d x \\
	&+2\sqrt{\alpha\psi(0)}\int_{J_u} \sqrt{\A ([u](z)\odot \nu(z)) \cdot ([u](z)\odot \nu(z)) } \d \H^{n-1}(z)\\
&+ 2\left(\int_{0}^1 \psi(t) \d t \right)\H^{n-1}(J_u) +\int_{\Omega} F(x,e(u),1)   \d x +\\
	& + \int_{J_u} F_{\infty}(z,[u]\odot \nu)  \d \H^{n-1}(z).
	\end{align*}
\text{}\\
\textbf{Step five:} \textit{$L^2$ convergence and $L^{\infty}$ bound}. By construction, it follows that $\|u\|_{L^{\infty}}\leq \|u\|_{L^{\infty}}$. We easily compute
	\begin{align*}
	\int_{\Omega} |u_{\e}-u|^2\d x &=\int_{A_{\vartheta \e}} |u_{\e}-u|^2\d x \leq C |A_{\vartheta\e}\|u\|_{\infty}^2 \rightarrow 0.
	\end{align*}
\end{proof}
\begin{remark}\label{crmk: ontrols on the gradient}
\rm{
Notice that, from \eqref{gradient} it follows also that
\begin{align*}
\int_{\Omega}|\nabla u_{\e}|\d x \leq C\left[|\Omega|(\|\nabla u\|_{\infty} +\|u\|_{\infty})+\int_{A_{\vartheta\e}} \frac{|S_{\vartheta \e}u (\bar{x})|}{\vartheta\e} \d x\right].
\end{align*}
Moreover, since $u$ is regular outside $J_u$ we can also see  that
	\begin{align*}
	|S_{\vartheta \e}u (\bar{x})|\leq |u(\bar{x}+\vartheta\e \nu) - u^+(\bar{x})|+ |u(\bar{x}-\vartheta\e \nu) - u^+(\bar{x})|+ |u^+(\bar{x}) - u^-(\bar{x})|
	\end{align*}
and 
	\begin{align*}
	 |u(\bar{x}+\vartheta\e \nu) - u^+(\bar{x})|\leq \int_{0}^{\e\vartheta(\bar{x})} |\nabla u(x)|\d x\leq \|\nabla u\|_{\infty}\e\vartheta.
	\end{align*}
Since
	\begin{align*}
	\int_{A_{\vartheta\e}}  \frac{|u^+(\bar{x}) - u^-(\bar{x})|}{\vartheta\e} \d x&=\int_{J_u}\d \H^{n-1}(y)\int_{-\vartheta\e}^{\vartheta\e} 
	\frac{|u^+(y) - u^-(y)|}{\vartheta\e} \d t\\
	&=2\int_{J_u}|[u]|\d \H^{n-1}(y)\leq 2\|u\|_{\infty}\H^{n-1}(J_u).
	\end{align*}
All this considered gives 
	\begin{align}\label{stima bound recovery 1}
\int_{\Omega} |\nabla u_{\e}|\d x \leq C,
\end{align}
for a constant $C$ that depends on $u$ and $\Omega$ only. Along the same line we can also obtain
	\begin{align*}
\int_{A_{\vartheta\e}} v_{\e}|\nabla u_{\e}|^2\d x &\leq C\left[\|\nabla u\|^2_{\infty}+ \alpha \int_{A_{\vartheta\e}} \e \frac{|u^+(\bar{x})-u^-(\bar{x})|^2}{\vartheta^2(\bar{x})\e^2} \d x\right]\\
&\leq  C\left[\|\nabla u\|^2_{\infty}+ \alpha \int_{J_u} \int_{-\vartheta(y)\e}^{\vartheta(y)\e} \frac{|u^+(y)-u^-(y)|^2}{\vartheta^2(y)\e} \d x\right]\leq  C,
	\end{align*}
for a constant $C$ that depends on $u$ and $\Omega$ only. In particular,
		\begin{align}\label{stima bound recovery 2}
\int_{\Omega} v_{\e}|\nabla u_{\e}|^2\d x \leq C,
\end{align}
for a constant $C$ that depends on $u$ and $\Omega$ only.
}
\end{remark}
\subsection{Recovery sequence for $u\in SBD^2(\Omega)\cap L^{\infty}(\Omega)$}
We provide an approximation Theorem based on the following two Theorems from \cite{iurlano2013fracture} and \cite{cortesani1999density}.
\begin{theorem}\label{thm IUR}[\cite{iurlano2013fracture}, Theorem 3.1]
Let $\Omega$ be an open bounded set with Lipschitz boundary and let $u\in GSBD^2(\Omega)\cap L^{\infty}(\Omega)$. Then there exists a sequence $\{u_k\}_{k\in \N}\subset SBV^2(\Omega;\R^n) \cap L^{\infty}(\Omega;\R^n) \cap W^{1,\infty}(\Omega \setminus S_k;\R^n) $ such that each $J_{u_k}$ is contained in the union $S_k$ of a finite number of closed, connected pieces of $C^1$-hypersurfaces and the following properties hold:
	\begin{align*}
	\text{a)} &\ \|u_k - u\|_{L^2} \rightarrow 0;\\
	\text{b)} & \ \|e(u_k) - e(u)\|_{L^2} \rightarrow 0;\\
	\text{c)} & \ \H^{n-1}(J_{u_k}\Delta J_u)=0;\\
	\text{d)} &\  \int_{J_{u_k} \cup J_u} \max\{ |u_k^{\pm} - u^{\pm}|, M\}\d \H^{n-1} \rightarrow 0 \ \ \ \text{for all $M\in \R^+$}.
	\end{align*}
Moreover $\|u_k\|_{L^{\infty}}\leq \|u\|_{L^{\infty}}$.
\end{theorem}
\begin{theorem}\label{thm CT}[\cite{cortesani1999density}, Theorem 3.1]
Let $\Omega$ be an open bounded set with Lipschitz boundary and let $u\in SBV^2(\Omega;\R^n)\cap L^{\infty}(\Omega;\R^n)$. Then there exists a sequence of function $\{u_k\}_{k\in \N} \subset SBV(\Omega;\R^n)$ such that
	\begin{itemize}
	\item[1)] $u_k\in W^{m,\infty}(\Omega\setminus \ov{J_{u_k}})$ for all $m\in \N$ and $\H^{n-1}((\ov{J_{u_k}} \cap \Omega) \setminus J_{u_k})=0$;
	\item[2)] The set $\ov{J_{u_k}}\cap \Omega$ is the the finite union of closed and pairwise disjoint $(n-1)$-simplexes intersected with $\Omega$;
	\item[3)] $\|u_k - u\|_{L^2} \rightarrow 0$;
	\item[4)] $\|\nabla u_k - \nabla u\|_{L^2}\rightarrow 0$;
	\item[5)] $\displaystyle \limsup_{k\rightarrow +\infty} \int_{\ov{A}\cap J_{u_k} } \varphi(x,u_k^+, u_k^-, \nu_k) \d \H^{n-1}(x) \leq \int_{\ov{A}\cap J_{u} } \varphi(x,u^+, u^-, \nu) \d \H^{n-1}(x)$
	\end{itemize}
where property $5)$ holds for every open set $A\subset \Omega$ and every upper semicontinuous function $\varphi:\Omega \times \R^n\times \R^n\times S^{n-1} \rightarrow [0,+\infty)$ such that
	\begin{align}
	\varphi(x,a,b,\nu)&=\varphi(x,b,a,-\nu) \ \ \ &\text{for all $x,a,b,\nu \in \Omega \times \R^n\times \R^n\times S^{n-1}$};\label{palle1}\\
	\limsup_{\substack{(y,a',b',\mu) \rightarrow (x,a,b,\nu)\\
	y\in \Omega  }} & \varphi(y,a',b',\mu)<+\infty \ \ \ &\text{for all $x,a,b,\nu \in \pa\Omega \times \R^n\times \R^n\times S^{n-1}$}.\label{palle2}
	\end{align}

\end{theorem}
About these results,  references of interest are \cite{Ch1, CC2017}.
As a consequence we obtain the following result:
\begin{proposition}\label{propo: density argument}
For any function $u\in SBD^2(\Omega)\cap L^{\infty}(\Omega)$ there exists a sequence $u_k\in \Cl$ such that $u_k\rightarrow u$ in $L^2$ and
	\[
	\lim _{k\rightarrow +\infty} \F(u_k,1) = \F(u,1).
	\]
Moreover $\|u_k\|_{\infty}\leq \|u\|_{\infty}$.
\end{proposition}
\begin{proof}
We will apply Theorem \ref{thm IUR} and \ref{thm CT} to improve the regularity of our sequence. We divide the proof in two steps.\\
\text{}\\
\textbf{Step one:} \textit{reduction to $SBV^2$}. Let $u\in SBD^2(\Omega)\cap L^{\infty}(\Omega) \subset GSBD^2(\Omega)\cap L^2(\Omega)$. Then, by applying Theorem \ref{thm IUR} we find a sequence of functions $\{w_k\}_{k\in \N}\subset SBV^2(\Omega)\cap L^{\infty}(\Omega)$ such that properties $a)$-$d)$ of Theorem \ref{thm IUR} hold. We have that
	\begin{align*}
	\mathcal{F}_k(w_k,1) =& \int_{\Omega} \A e(w_k)\cdot e(w_k) \d x \\
	&+2\sqrt{\alpha\psi(0)}\int_{J_{u_k}} \sqrt{\A ([w_k](z)\odot \nu_k(z)) \cdot ([w_k](z)\odot \nu_k(z)) } \d \H^{n-1}(z)\\
&+ 2\left(\int_{0}^1 \psi(t) \d t \right)\H^{n-1}(J_{w_k}) +\int_{\Omega} F(x,e(w_k),1)   \d x\\
	& + \int_{J_{w_k}} F_{\infty}(z,[w_k]\odot \nu_k)  \d \H^{n-1}(z).
	\end{align*}
In particular, because of property $b)$ we can infer that $\mathbb{A}e(w_k)\cdot e(w_k) \rightarrow \mathbb{A}e(u)\cdot e(u) $ in $L^1$ and that $F(x,e(w_k),1)\rightarrow F(x,e(u),1)$ where we exploited the fact that $F$ is a Lipschitz function (thanks to Remark \ref{rmk: Lipschitz function}). This allows us to write
	\begin{align}
	\lim_{k\rightarrow+\infty} & \int_{\Omega} \A e(w_k)\cdot e(w_k) \d x + 2\left(\int_{0}^1 \psi(t) \d t \right)\H^{n-1}(J_{w_k}) +\int_{\Omega} F(x,e(w_k),1)   \d x \nonumber\\
&= \int_{\Omega} \A e(u)\cdot e(u) \d x + 2\left(\int_{0}^1 \psi(t) \d t \right)\H^{n-1}(J_{u}) +\int_{\Omega} F(x,e(u),1)   \d x.\label{eqn:Pessoa}
	\end{align}
Since the function $u$ is $L^{\infty}$ and $\|w_k\|_{L^{\infty}}\leq \|u\|_{L^{\infty}}$ we have that (because of property $d)$ of our sequence)
	\begin{align}
	\int_{J_{w_k} \cup J_u} | [w_k] \odot \nu_k -  [u] \odot \nu | \d \H^{n-1}(z)  \leq& \int_{J_{w_k} \cup J_u} | ( [w_k] - [u] ) \odot \nu_k | \d\H^{n-1}(z)\nonumber \\
	&+\int_{  J_{w_k} \cup J_u} |[u] \odot (\nu-\nu_k) | \d \H^{n-1}(z) \nonumber\\
	 \leq& \int_{J_{w_k} \cup J_u } |[w_k]-[u]|\d \H^{n-1}(z) \nonumber\\	
	 &+\int_{  J_{w_k} \cap J_u} |[u] \odot (\nu-\nu_k) | \d \H^{n-1}(z)  \nonumber \\
	  &+\|u\|_{L^{\infty}}\H^{n-1}(J_{w_k}\Delta J_u) \rightarrow 0, \label{eqn:fernando}
	\end{align}
since $\nu_k=\nu$ $\H^{n-1}$-a.e. on $J_{w_k}\cap J_u$. The functions $\sqrt{\A M \cdot M}$ and $F_{\infty}$ are $1-$homogeneous and convex and thus Lipschitz on $M_{sym}^{n\times n}$ (Remark \ref{rmk: Lipschitz function}). Then
\begin{align*}
&|\sqrt{\A ([w_k](z)\odot \nu_k(z)) \cdot ([w_k](z)\odot \nu_k(z)) }  -\sqrt{\A ([u](z)\odot \nu(z)) \cdot ([u](z)\odot \nu(z)) } |\\
		& + | F_{\infty}(z,[w_k]\odot \nu_k)  - F_{\infty}(z,[u]\odot \nu)| \leq C | [w_k] \odot \nu_k - [u]\odot\nu|
\end{align*}		
that integrated over $J_{w_k}\cup J_u$ and passed  to the limit yields by \eqref{eqn:fernando}
 \begin{align}
  \lim_{k\rightarrow+\infty} &2\sqrt{\alpha\psi(0)}\int_{J_{w_k}} \sqrt{\A ([w_k](z)\odot \nu_k(z)) \cdot ([w_k](z)\odot \nu_k(z)) } \d \H^{n-1}(z)\nonumber\\
& \ \ \ \ \ + \int_{J_{w_k}} F_{\infty}(z,[w_k]\odot \nu_k)  \d \H^{n-1}(z)=\int_{J_{u}} F_{\infty}(z,[u]\odot \nu)  \d \H^{n-1}(z)\nonumber\\
& \ \ \ \ \ \ \ \ \ \ \ \ \ \ \ \  + 2\sqrt{\alpha\psi(0)}\int_{J_{u}} \sqrt{\A ([u](z)\odot \nu(z)) \cdot ([u](z)\odot \nu(z)) } \d \H^{n-1}(z).\label{eqn:saramago}
 \end{align}
By virtue of \eqref{eqn:Pessoa} and \eqref{eqn:saramago} we have produced a sequence $w_k$ such that $w_k\rightarrow u$ in $L^2$ and
	\begin{equation}\label{limione}
	\lim_{k\rightarrow +\infty} \F(w_k,1)=\F(u,1).
	\end{equation}
	\text{}\\
\textbf{Step two:} \textit{regularization to $\Cl$}. For any $w=w_k$ produced in step one we can produce, by applying Theorem \ref{thm CT}, a sequence $\{u_k\}_{k\in \N}$ such that $u_k\in \Cl$ and satisfying $1)-5)$ of Theorem \ref{thm CT}. In particular $u_k\rightarrow w$ in $L^2$ and, thanks to property $4)$ and $5)$ we obtain
	\begin{align*}
	\limsup_{k\rightarrow +\infty} & \int_{\Omega} \A e(u_k)\cdot e(u_k) \d x \\
	&+2\sqrt{\alpha\psi(0)}\int_{J_{u_k}} \sqrt{\A ([u_k](z)\odot \nu_k(z)) \cdot ([u_k](z)\odot \nu_k(z)) } \d \H^{n-1}(z)\\
&+ 2\left(\int_{0}^1 \psi(t) \d t \right)\H^{n-1}(J_{u_k}) +\int_{\Omega} F(x,e(u_k),1)   \d x + \int_{J_{u_k}} F_{\infty}(z,[u_k]\odot \nu_k)  \d \H^{n-1}(z)\\
	& \leq \int_{\Omega} \A e(w)\cdot e(w) \d x +2\sqrt{\alpha\psi(0)}\int_{J_{w}} \sqrt{\A ([w](z)\odot \nu(z)) \cdot ([w](z)\odot \nu(z)) } \d \H^{n-1}(z)\\
&+ 2\left(\int_{0}^1 \psi(t) \d t \right)\H^{n-1}(J_{w}) +\int_{\Omega} F(x,e(w),1)   \d x+ \int_{J_{w}} F_{\infty}(z,[w]\odot \nu )  \d \H^{n-1}(z)\\
	\end{align*}
where we have exploited the fact that $F$ is Lipschitz continuous (as in step one) and that the function
	\[
	\varphi(z,a,b,\nu):=2\sqrt{\alpha\psi(0)}\sqrt{\A ((a-b)\odot \nu ) \cdot ( (a-b)\odot \nu ) } + F_{\infty}(z,(a-b)\odot \nu)
	\]
is always positive (due to the hypothesis \ref{hyp on F} on $F$) and satisfies assumptions \eqref{palle1}, \eqref{palle2}. By possibly passing to the truncated $\hat{u}_k(x)=\max\{u_k(x),\|w\|_{\infty}\}$ the above inequality is preserved together with the condition $\|\hat{u}_k\|_{\infty}\leq \|w\|_{\infty}\leq\|u\|_{\infty}$. By taking into account Theorem \ref{thm: liminf inequality}, it is deduced that
	\begin{equation}\label{limitwo}
	\lim_{k\rightarrow +\infty} \F(\hat{u}_k,1)=\F(w,1).
	\end{equation}
By combining \eqref{limione}, \eqref{limitwo} with a diagonalization argument on $\hat{u}_k, w_j$ we can produce the sought sequence.
\end{proof} 
We are thus in the position to state the $\limsup$ upper bound and provide a recovery sequence for functions $u\in L^{\infty}(\Omega;\R^n)$.
\begin{theorem}\label{thm: limsup inequality}
Let $\{\e_j\}_{j\in \N}$ be a vanishing sequence of real numbers. Then, for any $u\in SBD^2(\Omega)\cap L^{\infty}(\Omega)$ there exists a subsequence $\{\e_{j_k}\}_{k\in \N}\subset \{\e_{j}\}_{j\in \N}$ and a sequence of function $(u_{k},v_{k})\in H^1(\Omega;\R^n) \times V_{\e_{j_k}}$ such that
	\[
	\e_{j_k} \rightarrow 0, \ \ \ \|u_k-u\|_{L^2}\rightarrow 0, \ \ \ \|v_k-1\|_{L^2}\rightarrow 0
	\]
and
	\[
	\lim_{\k\rightarrow +\infty} \F_{\e_{j_k} }(u_{k},v_{k}) = \F(u,1).
	\]
Moreover $\|u_k\|_{{\infty}}\leq \|u\|_{{\infty}}$.
\end{theorem}
\begin{proof}
We prove that, for any $k>0$ there exists an $\e_{j_k}$ and $(u_{k},v_{k})\in H^1(\Omega;\R^n)\times V_{\e_{j_k}}$ such that
	\begin{equation}\label{alakazam}
	\|v_{k} - 1\|_{L^1} + \|u_{k} - u\|_{L^2} + |\F_{\e_{j_k}}(u_{k} ,v_{k}) - \F(u,1)| +\e_{j_k} \leq \frac{1}{k}. 
	\end{equation}
This would complete the proof. According to Proposition \ref{propo: density argument}, for any fixed $k>0$ we can find $w\in \Cl$ with $\|w\|_{\infty}\leq \|u\|_{\infty}$ such that
	\begin{equation}\label{kadabra}
	|\F(w,1)-\F(u,1)|+\|w-u\|_{L^2}\leq\frac{1}{2k}.
	\end{equation}
The sequence $(u_{\e},v_{\e})\in H^1(\Omega)\times V_{\e}$ as defined in \eqref{eqn: rcvry v regular}, \eqref{eqn: rcvry u regular} (thanks to Proposition \ref{propo: approx of nice function}) provides
	\[
	\lim_{\e\rightarrow 0} \F_{\e}(u_{\e},v_{\e})=\F(w,1),
	\]
with $w \in \Cl$, and satisfies $\|u_{\e}\|_{L^{\infty}}\leq \|w\|_{L^{\infty}}\leq \|u\|_{L^{\infty}}$. In particular, we can find an $\e_{0}(k)$ such that
	\[
	\|u_{\e}-w\|_{L^2}+\|v_{\e}- 1\|_{L^2} + |\F_{\e} (u_{\e},v_{\e})-\F(w,1)|\leq \frac{\delta}{4} \ \ \ \text{for all $\e<\e_0$}.
	\]
We can select an $\e_{j_k}<\e_0(k)$, since $\e_j$ is vanishing,  such that 
	\begin{equation}\label{abra}
	\|u_{\e_{j_k}}-w\|_{L^2}+\|v_{\e_{j_k}}- 1\|_{L^2} + |\F_{\e_{j_k}} (u_{\e_{j_k}},v_{\e_{j_k}})-\F(w,1)|+\e_{j_k}\leq \frac{1}{2k}.
	\end{equation}
By combining \eqref{abra} and \eqref{kadabra} and by setting $u_k=u_{\e_{j_k}}$, $v_k=v_{\e_{j_k}}$ we obtain \eqref{alakazam}. 
\end{proof}

\section{Compactness result and minimum problem}\label{CER}
\subsection{Compactness}\label{sct Comp}
This section is devoted to the proof of Theorem \ref{main thm: comp}. 
\begin{proof}[Proof of Theorem \ref{main thm: comp}]
From \cite[Theorem II.2.4]{suquet1981equations} we obtain $L^1$-compactness from uniform $BD$-boundedness. In particular notice that, if $(u_{\e},v_{\e})\in H^1(\Omega;\R^n) \times V_{\e}$ satisfies 
$\sup_{\e}\{\|u_{\e}\|_{L^1}+\F(u_{\e},v_{\e})\}<+\infty$, 
then, thanks to Proposition \ref{prop: bnd energy AT}, we have
$\sup_{\e}\{\mathcal{W}(u_{\e},v_{\e})\}<+\infty$. By then arguing as in the proof of Proposition \ref{prop: information on convergence} we can retrieve relations  \eqref{eqn: inside out}  and \eqref{eqn: inside in}  that imply $
	\sup_{\e>0}\left\{ \int_{\Omega} |e(u_{\e})|\d x\right\}<+\infty$.
This, combined with the uniform $L^1$ upper bound on $u_{\e}$ gives a uniform bound on the $BD$ norm leading to $L^1$ compactness of $u_{\e}$. Moreover $
	\sup_{\e}\{\mathcal{W}(u_{\e},v_{\e})\}\geq \frac{1}{\e}\int_{\Omega} \psi(v_{\e}) \d x$
implying that $\psi(v_{\e}) \rightarrow 0$ in measure and thus $v_{\e}\rightarrow 1$ in measure. Then there is a subsequence converging to $1$ almost everywhere and due to the boundedness of $v_{\e}$ we have (up to a subsequence) $v_{\e} \rightarrow 1$ in $L^1$. In particular we have shown that, up to a subsequence (not relabeled), it holds $u_{\e} \rightarrow u$, $v_{\e}\rightarrow 1$ in $L^1$ and $Eu_{\e} \rightharpoonup^* Eu$. By applying Proposition \ref{prop: information on convergence} we obtain $u\in SBD^2(\Omega)$.

\end{proof}

\subsection{Statement of the minimum problem}\label{Exist}
We discuss the issue of existence of minimizers under Dirichlet boundary condition. We restrict ourselves to smooth boundary data on an open bounded set having smooth boundary.
%

From now on the set $\Omega$ will  be assumed to be an open bounded set with $C^1$ boundary.  Assume that $\mathbb{A},F,\psi$ are as in \ref{sct:settings}.  On the potential $F$ we require additionally that 
	\begin{align}
	& \text{•} \ F(x,\cdot,v) \ \ \  \text{is convex for all $(x,v)\in \Omega\times [0,1]$}\label{1.a}\\
 &\text{•} \ \rho=\sup\left\{\frac{|F(x,M,v)-F(y,M,v)|}{|M||x-y|} \ | \ x,y\in \Omega, \ M\in \M_{sym}^{n\times n} , \ v\in [0,1]\right\}<+\infty.\label{1.b}
	\end{align}
and that, having fixed, for $s,t\in (0,1)$,
    \[
     \omega_F(s;t):=\sup\left\{\frac{|F(x,M,s)-F(x,M,t)|}{|M|} \ | \ (x,M)\in \Omega\times \M_{sym}^{n\times n} \right\}
    \]
it holds 
\begin{equation}\label{1.c}
    \lim_{s\rightarrow t}  \omega_F(s;t) =0 \ \ \ \ \text{for all $t\in (0,1)$}.
\end{equation}
Consider, for fixed $d\in \R$ the following infimum problems
	\begin{align*}
	\gamma_{\e}&:= \inf\left\{ \F_{\e}(u,v) \ |  \ u=f, \ v=1 \ \text{on $\pa \Om$},  \ (u,v)\in H^{1}(\Omega;\R^n)\times V_{\e}, \ \|u\|_{L^{\infty}}\leq d \right\},\\
	\gamma_{0}&:=\inf\left\{ \F(u,1) + \mathcal{R}(u,f)\ |  \ u\in SBD^2(\Omega), \ \|u\|_{L^{\infty}}\leq d \right\},
	\end{align*}
where 
	\begin{equation}
	\begin{split}
	\mathcal{R}(u,f):=& \int_{\pa \Omega} F_{\infty}(z,[u-f]\odot \nu) \d \H^{n-1}(z)+b\H^{n-1}( \{x\in\pa \Omega \ | \ u(x) \neq f(x) \}) \\
	&+a\int_{\pa \Omega} \sqrt{\mathbb{A}[u-f]\odot \nu \cdot [u-f]\odot \nu}\d \H^{n-1}(z).
	\end{split}
	\end{equation}
Notice that the additional term $\mathcal{R}(\cdot,f)$ is the price that a function has to pay in order to detach from the boundary datum $f$ on $\pa \Omega$. Then the following Theorem holds true.

\begin{theorem}\label{thm: existence}
For every $\e> 0$ there exists minimizers $(u_{\e},v_{\e})$ for $\gamma_{\e}$. Moreover 
	\begin{equation}\label{eqn: limit min}
	\lim_{\e\rightarrow 0} \gamma_{\e}=\gamma_0
		\end{equation}
and any accumulation point of $\{(u_{\e},v_{\e})\}_{\e>0}$ is of the form $(u_0,1)$ with $u_0$ a minimum for $\gamma_0$.
\end{theorem}
This implies that, by combining the compactness Theorem \ref{main thm: comp} with Theorem \ref{thm: existence}, we can prove the following corollary.
\begin{corollary}
There exists at least a minimizer for the problem $\gamma_0$.
\end{corollary}
The proof of Theorem \ref{thm: existence} follows by showing that the problem $\gamma_{\e}$ ($\Gamma$)-converges to the problem $\gamma_0$. While it is easy to show that $\liminf \gamma_{\e}\geq \gamma_0$ in order to prove the $\limsup$ inequality we have to exhibit a recovery sequence with fixed boundary datum. Note that this approach to handle the boundary datum was proposed in \cite{AmbrosioElast} in the anti-plane case, though without a formal proof.

The arguments we used to address existence results should be considered as a title of example in order to introduce and formalize an approach based on the extension of the domain $\Omega$. For this reason, its generality is restricted. In particular, for the sake of simplicity we restrict ourselves to smooth boundary data considered on domain with smooth boundary. We are however confident that, with a refined analysis of the surgeries, one can carry out a more general statement involving $H^{1/2}$ boundary data defined on pieces of the boundary $\pa \Omega$ of a Lipschitz domain.
\subsection{Recovery sequence with prescribed boundary condition}\label{sbsct: rcvry with fixed bndry}
We now proceed to show how to recover the energy of a function $u\in \Cl$ by making use of function $u_j\in H^1(\Omega;\R^n)$ with smooth boundary data $f\in C^1(\pa \Omega;\R^n)$. At the very end, by making use of Theorem \ref{thm IUR} and  \ref{thm CT} we show that we can recover the energy $\F(u,1)+\mathcal{R}(u,f)$ of any $u\in SBD^2(\Omega)$.
\begin{figure}\label{fig: surgery}
\includegraphics[scale=0.7]{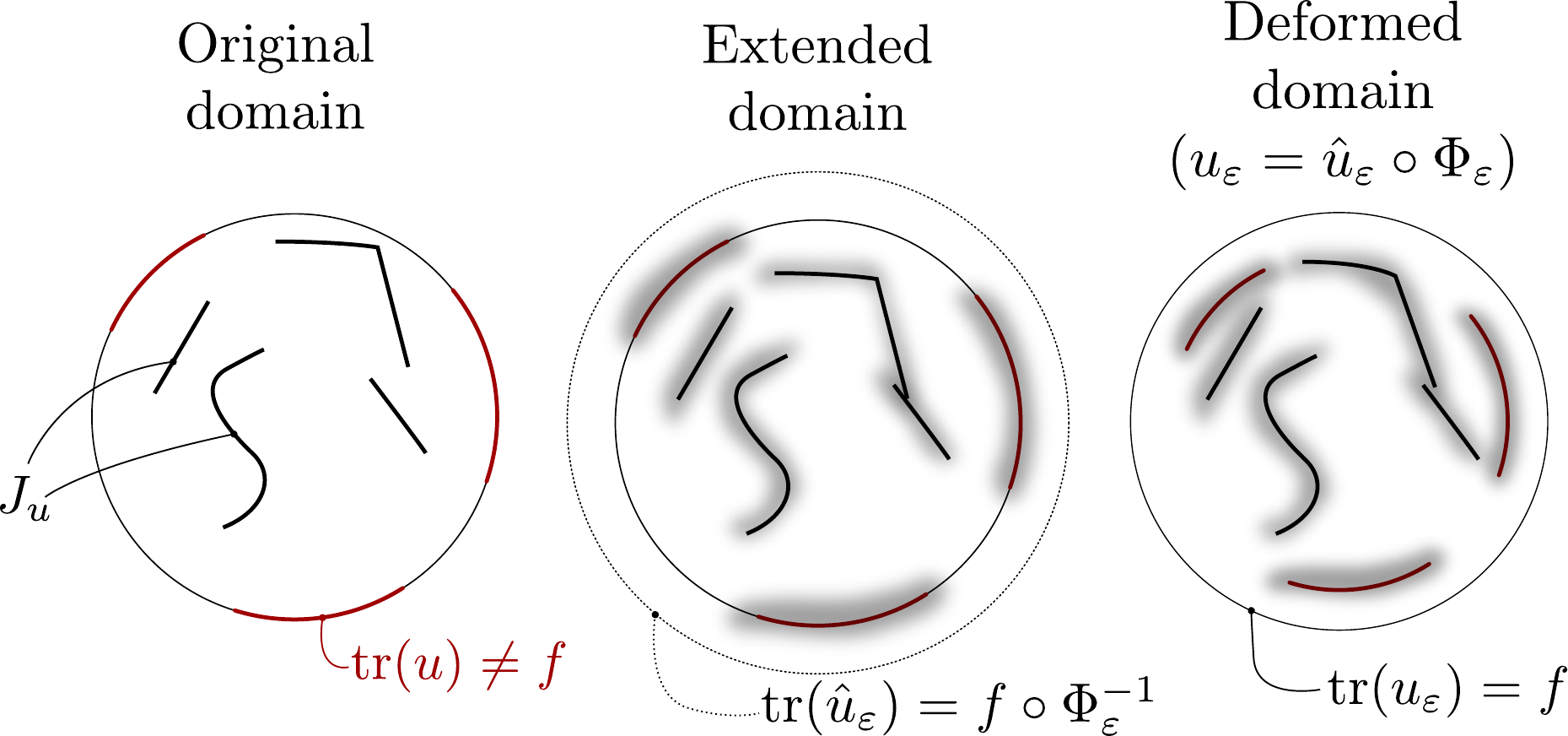}\caption{In red is depicted the region where $\trace(u)\neq f$. After the normal extension we can see that the region $\{\trace(u)\neq f\}$ has become just part of $J_{\hat{u}}$. We then consider a recovery sequence $(\hat{u}_{\e},\hat{v}_{\e})$ defined as in \eqref{eqn: rcvry v regular} and \eqref{eqn: rcvry u regular}. The grey part represents the region where the damage variable $\hat{v}_{\e}<<1$. Finally, by composing $u_{\e},v_{\e}$ with the diffeomorphism $\Phi_{\e}$ provided by Proposition \ref{geometry}, we go back to our domain $\Omega$ by preserving the boundary condition. This operation does not affect in a significant way the energy and we asymptotically recover the sharp energy, which also accounts for $\mathcal{R}(u,f)$ (that comes exactly  from those regions where $\{\trace(u)\neq f\}$). }
\end{figure}
We briefly sketch the proof for regular functions before moving to the technical part. As depicted in Figure \ref{fig: surgery} it might happen that $\trace(u)\neq f$ on $\pa \Omega$. To handle also this situation, which represents the main challenge of our proof, we first extend normally our $u\in \Cl$ into a $\hat{u}$ defined on a slightly larger $\hat{\Omega}\supset \Omega$ in a way that does not destroy the regularity of $u$. In this way, any region on $\pa \Omega$ where $\trace(u)\neq f$ becomes the jump region of $\hat{u}$ and it is well contained in the extended domain. Thus we can proceed to define the recovery sequence as in  \eqref{eqn: rcvry v regular} and \eqref{eqn: rcvry u regular}. Such a recovery sequence coincides with $u$ far enough from the jump set and this allows us to deduce a strong control on the energy in the strip $\hat{\Omega}\setminus \Omega$. This normal extension further allows us to deduce that along the level set $E_t= \{d(x,\pa\Omega)=t\}$ (for suitable $t$) we have that $u_{\e}\Big{|}_{E_t}=f$. Then, by applying a smooth diffeomorphism, that we are able to control in terms of $\e$, we shrink back our extended domain onto $\Omega$ so that $E_t\mapsto \pa \Omega$ and this guarantees that the whole boundary condition is satisfied.\\
\text{}\\
We start with the following technical Lemma that will provide us the required family of diffeomorphisms. Let us recall that we are denoting by $P:(\pa \Omega)_{\delta}\rightarrow \pa\Omega$ the orthogonal projection onto $\pa \Omega$ well defined on any tubular neighbourhood $(\pa \Omega)_{\delta}$ of $\pa \Omega$ small enough. Moreover we are always considering the outer unit normal $\nu_{\Omega}:\pa \Omega\rightarrow S^{n-1}$ and we recall that, with the notation $\dist(x,\pa\Omega)$, we are always meaning the \textit{signed distance}
    \[
    \dist(x,\pa \Omega):= (x-P(x))\cdot \nu_{\Omega}(P(x))
    \]
well defined on small tubular neighbourhoods around $\pa \Omega$.
\begin{lemma}\label{geometry}
Let $\Omega$ be an open bounded set with $C^1$ boundary and consider $(\pa \Omega)_{\delta}$  any fixed tubular neighborhood of $\pa \Omega$  where the projection operator $x\mapsto P(x)\in \pa \Omega$ is well defined. Let also $(\pa \Omega)_{\e L}$ be another tubular neighborhood where $L>0$ is any real constant and set $\Omega_{\e}=\Omega\cup (\pa \Omega)_{\e L}$. Then there exists a family of diffeomorphism $\{\Phi_{\e} :\Omega_{\e} \rightarrow \Omega\}_{\e>0}$ such that 
	\begin{align}
	\lim_{\e\rightarrow 0} \sup_{x\in \Omega} \{|J\Phi_{\e}(x)|\}&=1,\label{diffeo1}\\
	 \lim_{\e\rightarrow 0}\sup_{x\in \Omega_{\e}}\{|J\Phi_{\e}^{-1}(x)|\}&=1,\label{diffeo2}\\
	 |\nabla \Phi_{\e}^{-1}(x)-\Id|+ |\nabla \Phi_{\e}(x)-\Id| &\leq C\e \label{diffeo3}
	\end{align}
where $C$ depends on $\Omega$, $L$ and $\delta$ only. Moreover 
\begin{align*}
P(\Phi_{\e}(x))=P(\Phi_{\e}^{-1}(x))&=P(x) \ \ \ \ \text{on $(\pa \Omega)_{\e L }\cup (\pa \Omega)_{\delta}$},\\ 
\Phi_{\e}(x)&=x \ \ \ \ \text{on $\Omega\setminus(\pa \Omega)_{\delta}$},
\end{align*}
and $\Phi_{\e}^{-1}(\pa \Omega_{\e})=\pa \Omega ,\  \ \Phi_{\e}(\pa \Omega)=\pa \Omega_{\e},$
\end{lemma}
\begin{proof}
\begin{figure}\label{fig: diffeo}
\includegraphics[scale=0.8]{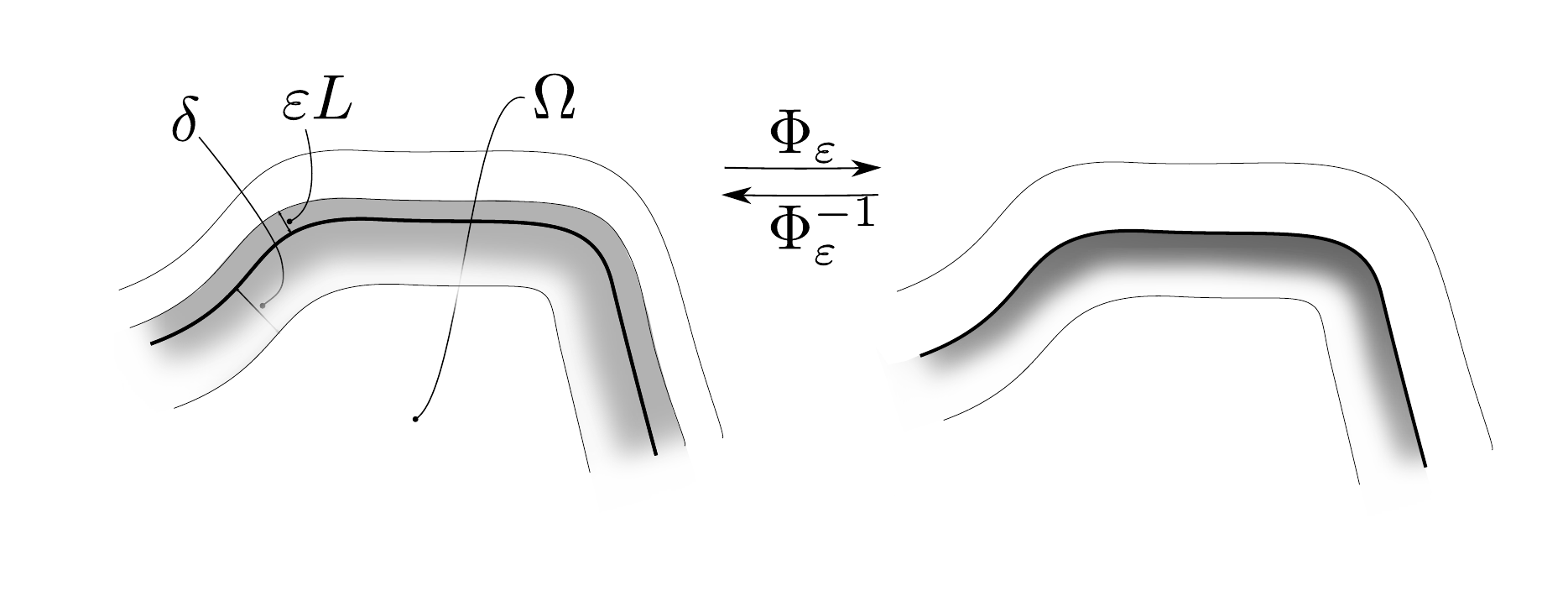}\caption{We shrink the region $(\pa \Omega)_{\e L} \cup [(\pa \Omega)_{\delta}\cap\Omega]$ onto $(\pa \Omega)_{\delta} \cap \Omega$ through $\Phi_{\e}$ by gently pushing the set along $\nu_{\Omega}$ with a strength that decays in $\dist(x,\pa \Omega)$ fast enough so that $\Phi_{\e}(x)=x$ on $\Omega\setminus (\pa \Omega)_{\delta}$.} 
\end{figure}
Consider the diffeomorphism, depicted in Figure \ref{fig: diffeo}:
	\begin{equation}
	\Phi_{\e}(x);=\left\{ 
	\begin{array}{ll}
	x \ &\ \text{if $x\in \Omega \setminus (\pa\Omega)_{\delta}$}\\
	x+\nu_{\Omega}(P(x))\frac{(\delta+\dist(x,\pa \Omega)}{\delta} \e L&\ \text{if $x\in \Omega \cap (\pa\Omega)_{\delta}$}
	\end{array}\right.
	\end{equation}
with inverse
	\begin{equation}
	\Phi_{\e}^{-1}(x);=\left\{ 
	\begin{array}{ll}
	x \ &\ \text{if $x\in \Omega_{\e} \setminus (\pa\Omega)_{\delta}$}\\
	x-\nu_{\Omega}(P(x))\frac{(\delta+\dist(x,\pa \Omega))}{\delta+\e L}\e L  &\ \text{if $x\in \Omega_{\e} \cap \Omega_{\delta}$}.
	\end{array}\right.
	\end{equation}
It is straightforward that $\Phi_{\e}(\pa \Omega)=\pa \Omega_{\e}$, $\Phi_{\e}^{-1}(\pa \Omega_{\e})=\pa \Omega $ and $\Phi_{\e}(x)=x$ on $\Omega\setminus(\pa \Omega)_{\delta}$. Moreover
	\[
\nabla \Phi_{\e}(x)= \Id + \nabla \nu_{\Omega}(P(x))\frac{(\delta + \dist(x,\pa\Omega))}{\delta}\e L+ \nu_{\Omega}(P(x))\otimes \nu_{\Omega}(P(x)) \frac{\e L}{\delta} 	
	\]
in particular the desired convergences \eqref{diffeo1}, \eqref{diffeo2} follow together with \eqref{diffeo3}.
\end{proof}
\begin{proposition} 
Let $\Omega$ be an open bounded set with $C^1$ boundary and let $f\in C^1(\pa \Omega;\R^n)$. For every $u\in \Cl$ there exists a sequence $(u_{\e},v_{\e})\in H^1(\Omega)\times V_{\e}$ such that $(u_{\e},v_{\e})\to (u,1)$ in $L^2$, $u_{\e}=f,  \ v_{\e}=1$ on $\pa \Omega$, 
and 
	\[
	\F_{\e}(u_{\e},v_{\e})\rightarrow \F(u,1)+\mathcal{R}(u,f)\quad \mbox{as}\quad\varepsilon\to0.
	\]
Moreover $\|u_{\e}\|_{\infty}\leq \|u\|_{\infty}$.
\end{proposition}

\begin{proof}
By virtue of Remark \ref{trick} we can always assume that $\ov{J_u}\cap \Omega\subset \Omega$. In particular we can find a $\delta>0$  that depends only on $\Omega$ and $u$ and such that $(\pa \Omega)_{\delta}\cap \ov{J_u} =\emptyset$.  We first define the extension $\hat{u}$ of $u\in \Cl$ as
	\begin{equation}\label{extended u}
	\hat{u}:=\left\{
	\begin{array}{ll}
	u(x) \ & \ \text{for $x\in \Omega$}\\
	f(P(x)) \ & \ \text{for $x\in (\pa \Omega)_{\delta}\setminus \Omega$}.
	\end{array}
	\right.
	\end{equation}
where $P$ denotes the orthogonal projection onto $\pa \Omega$ which is well defined on $(\pa \Omega)_{\delta}$ for $\delta$ small enough.
Then, having in mind Remarks \ref{Reg} and \ref{rmk costraint ve}, for any $\e>0$ we define $(\hat{u}_{\e},\hat{v}_{\e})\in H^1(\Omega \cup(\pa \Omega)_{\delta};\R^n)\times V_{\e}$ as in \eqref{eqn: rcvry v regular} and \eqref{eqn: rcvry u regular} with the $\vartheta$ provided in step four of the proof of Proposition \ref{propo: approx of nice function} (clearly we mean $V_{\e}$ referred to the domain $\Omega \cup(\pa \Omega)_{\delta}$ which is here not explicitly denoted in order to enlighten the notation). Notice that
	\[
	[\hat{u}]\H^{n-1}\llcorner J_{\hat{u}} =[u]\H^{n-1}\llcorner J_u+[\trace(u)-f]\H^{n-1}\llcorner \pa \Omega.
	\]
According to the definition of $\hat{u}_{\e}$ in \eqref{eqn: rcvry u regular}, we can see that $\hat{u}_{\e}(x)=\hat{u}(x)$ for all $x$ such that $d(x,J_{\hat{u}})> L_0 \e$ for an $L_0$ depending on $u$ only. In particular we can choose a suitable $L>0$ so to guarantee that $\hat{u}_{\e}(x)=\hat{u}(x)=f(P(x))$, $\hat{v}_{\e}(x)=1$ for all $x\in  [(\pa \Omega)_{\delta}\setminus (\pa \Omega)_{L\e}] \setminus \Omega $. We now apply our Lemma \ref{geometry} to $\Omega$ with the tubular neighborhoods $(\pa \Omega)_{\delta}, (\pa \Omega)_{\e L}$ to produce a family of diffeomorphism $\{\Phi_{\e}:\Omega\rightarrow \Omega\cup(\pa\Omega)_{\e L}=\Omega_{\e}\}$. By virtue of the computations in Remark \ref{crmk: ontrols on the gradient} and in particular due to \eqref{stima bound recovery 1} and \eqref{stima bound recovery 2} we can deduce also
	\begin{align}
	\int_{\Omega_{\e}} |\nabla \hat{u}_{\e}(x)|\d x+\int_{\Omega_{\e}} \hat{v}_{\e}|\nabla \hat{u}_{\e}|^2\d x&\leq C,\label{final control of gradient}
	\end{align}	 
	for a constant $C>0$ that depends on $\Omega$ and $u$ only (and that in the sequel may vary from line to line), while it is clear that the same computation performed in the proof of Proposition \ref{propo: approx of nice function} leads to
	\begin{align}\label{energy conv}
	\lim_{\e\rightarrow 0} \F_{\e}(\hat{u}_{\e},\hat{v}_{\e};\Omega_{\e})=\F(u,1)+\mathcal{R}(u,f).
	\end{align}
By making use of this facts we proceed to define $(u_{\e},v_{\e})\in H^1(\Omega;\R^n)\times V_{\e}$ by simply shrinking our domain $\Omega_{\e}$ into $\Omega$ throughout $\Phi_{\e}$. More precisely
	\begin{align*}
	u_{\e}(x)&:= \hat{u}_{\e}(\Phi_{\e}(x)), \ \ \ \   v_{\e}:=\hat{v}_{\e}(\Phi_{\e}(x)).
	\end{align*}
Notice that for $x\in \pa \Omega$ we have $\Phi_{\e}(x)\in \pa \Omega_{\e} \setminus \Omega\subset (\pa \Omega)_{\delta}\setminus \Omega$ and that $ P(\Phi_{\e}(x))=P(x)=x$ for all $x\in \pa \Omega $. 
Hence
	\[
	u_{\e}(x)=\hat{u}_{\e}(\Phi_{\e}(x))=f(P((\Phi_{\e}(x)))=f(x), \ \ \ \ v_{\e}(x)=\hat{v}_{\e}(\Phi_{\e}(x))=1 \ \ \ \ \text{for all $x\in \pa \Omega$}.
	\]
We underline that, as in Remark \ref{rmk costraint ve}, we are once again neglecting a possible factor (asymptotically equal to $1$) in front of $v_{\e}$ that might be needed in order to comply with the constraint $|\nabla v_{\e}(x)|\leq 1/\e$. Up to this carefulness we can infer $( u_{\e}, v_{\e})\in H^1(\Omega;\R^n)\times V_{\e}$. 
The $L^1$ convergence is immediately derived from the easy relations
    \begin{eqnarray*}
    \int_{\Omega}|u_\e - u|\d x&\leq& \int_{\Omega}|\hat{u}_\e(\Phi_\e(x)) - \hat{u}(\Phi_\e(x))|\d x + \int_{\Omega}|\hat{u}(\Phi_\e(x)) - u(x)|\d x,\\
    \int_{\Omega}|\hat{u}(\Phi_\e(x)) - u(x)|\d x&\leq& C\e (\|u\|_{\infty}+\|\nabla u\|_{\infty}),
    \end{eqnarray*} 
also holding for the function $v_\e$.
     It remains to show that the energy of the pairs $( u_{\e},v_{\e})$ is converging to $\F(u,1)+\mathcal{R}(u,f)$. From
	\[
	\nabla u_{\e}(x)-\nabla \hat{u}_{\e}(\Phi_{\e}(x))=(\nabla \Phi_{\e}(x) -\Id)\nabla \hat{u}_{\e}(\Phi_{\e}(x)),
	\]
and thanks to \eqref{diffeo3} we get
	\begin{align*}
	|\nabla  u_{\e}(x)-\nabla \hat{u}_{\e}( \Phi_{\e}(x))|\leq C\e|\nabla \hat{u}_{\e}(\Phi_{\e}(x))|,
	\end{align*}
for a constant  $C$ depending on $\Omega$ and $u$ only. In particular,
	\begin{align*}
	&\left| \int_{\Omega}v_{\e}(x)\mathbb{A}[ e(u_{\e})(x)-e(\hat{u}_{\e})(\Phi_{\e}(x))]\cdot e(\hat{u}_{\e})(\Phi_{\e}(x))\d x \right|\\
		&+\left| \int_{\Omega}v_{\e}(x)\mathbb{A}[ e(u_{\e})(x)-e(\hat{u}_{\e})(\Phi_{\e}(x)\cdot[e(u_{\e})(x)-e(\hat{u}_{\e})(\Phi_{\e}(x))]\d x \right|\\
	& \leq C\e \int_{\Omega}\hat{v}_{\e}(\Phi_{\e}(x))|\nabla \hat{u}_{\e} (\Phi_{\e}(x))|^2\d x= C\e \int_{\Omega_{\e}}\hat{v}_{\e}(x)|\nabla \hat{u}_{\e} (x)|^2 |J \Phi_{\e}^{-1}(x)|\d x,
	\end{align*}
which vanishes due to \eqref{diffeo2} and \eqref{final control of gradient}. Along the same lines and by exploiting Remark \ref{rmk: Lipschitz function} combined with hypothesis \eqref{1.a} and item 3) in \ref{hyp on F} on $F$ we get
\begin{align*}
\int_{\Omega}  |F(x,e(u_{\e})(x),v_{\e}(x))& - F(x,e(\hat{u}_{\e})(\Phi_{\e}(x)),v_{\e}(x))|\d x\\
&\leq \ell \int_{\Omega}|e(u_{\e})(x)-e(\hat{u}_{\e}(\Phi_{\e}(x))|\d x\\
& \leq C\ell \e  \int_{\Omega}|\nabla \hat{u}_{\e}(\Phi_{\e}(x))|\d x\\
&\leq C\e  \int_{\Omega_{\e}}|\nabla \hat{u}_{\e}(x)||J\Phi_{\e}^{-1}(x)|\d x\rightarrow 0,
 \end{align*}
once again due to \eqref{diffeo2} and \eqref{final control of gradient}. On the other hand, by exploiting \eqref{1.b} we can infer that
	\begin{align*}
\int_{\Omega}  |F(x,e(\hat{u}_{\e})(\Phi_{\e}(x)),\hat{v}_{\e}(x))& - F(\Phi_{\e}(x),e(\hat{u}_{\e})(\Phi_{\e}(x)),\hat{v}_{\e}(x))|\d x\\
&\leq \rho \int_{\Omega}|\Phi_{\e}(x)-x||e(\hat{u}_{\e}(\Phi_{\e}(x))|\d x\\
&\leq C\e  \int_{\Omega_{\e}}|\nabla \hat{u}_{\e}(x)||J\Phi_{\e}^{-1}(x)|\d x\rightarrow 0.
 \end{align*}
 In particular, all this considered we can conclude that
 	\begin{align*}
 	\lim_{\e\rightarrow 0} \F_{\e}( u_{\e}, v_{\e};\Omega)=\lim_{\e\rightarrow 0}& \int_{\Omega}v_{\e}(x) \mathbb{A} e(\hat{u}_{\e})(\Phi_{\e}(x))\cdot e(\hat{u}_{\e})(\Phi_{\e}(x)) \d x\\
 	&+\int_{\Omega}F(\Phi_{\e}(x),e(\hat{u}_{\e})(\Phi_{\e}(x)),v_{\e}(x))\d x+ \int_{\Omega}\frac{\psi( v_{\e}(x))}{\e} \d x\\
 	=\lim_{\e\rightarrow 0}& \int_{\Omega_{\e}}|J\Phi_{\e}^{-1}| \hat{v}_{\e}\mathbb{A} e(\hat{u}_{\e})\cdot e(\hat{u}_{\e}) \d x\\
 	&+\int_{\Omega_{\e}}|J\Phi_{\e}^{-1}|F(x,e(\hat{u}_{\e}),\hat{v}_{\e})\d x+ \int_{\Omega_{\e}}|J\Phi_{\e}^{-1}|\frac{\psi(\hat{v}_{\e}(x))}{\e} \d x\\
 	=\lim_{\e\rightarrow 0}& \F_{\e}(\hat{u}_{\e},\hat{v}_{\e};\Omega_{\e})=\F(u,f)+\mathcal{R}(u,f),
 	\end{align*}
 where we exploited \eqref{diffeo2} and \eqref{energy conv}. Notice that the condition $\|u_{\e}\|_{\infty}\leq \|u\|_{\infty}$ is preserved by construction and thus $(u_{\e}, v_{\e})$ provide the desired sequences.
\end{proof}
\begin{proposition} \label{approx to Cl}
Let $\Omega$ be an open bounded set with $C^1$ boundary and fix a smooth boundary data $f\in C^1(\pa \Omega;\R^n)$. For any function $u\in SBD^2(\Omega)\cap L^{\infty}(\Omega)$ there exists a sequence $u_k\in \Cl$ such that $u_k\rightarrow u$ in $L^2$ and
	\[
	\limsup _{k\rightarrow +\infty} \F(u_k,1)+\mathcal{R}(u_k,f) \leq \F(u,1)+\mathcal{R}(u,f)
	\]
Moreover $\|u_k\|_{\infty}\leq \|u\|_{\infty}$.
\end{proposition}
\begin{proof}
Consider $\hat{\Omega}\supset \Omega$ be a slightly larger domain and consider $w\in H^1(\Omega;\R^n)$ to be such that $w\Big{|}_{\pa \Omega}=f$. Consider the extension $\hat{u}:=u\ca_{\Omega}+w\ca_{\hat{\Omega}\setminus\Omega }\in SBD^2(\hat{\Omega})$. Then, by virtue of Proposition \ref{propo: density argument}, we can find a sequence $\hat{u}_k\in \text{Cl}(\hat{\Omega};\R^n)$ such that
	\[
	\F(\hat{u}_k,1;\hat{\Omega})\rightarrow \F(\hat{u},1;\hat{\Omega})=\F(u,1)+\mathcal{R}(u,f)+\int_{\hat{\Omega}\setminus \Omega} \A e(w)\cdot e(w) \d x +\int_{\hat{\Omega}\setminus \Omega} F(x,e(w),1)   \d x
	\]
and with $\|\hat{u}\|_{\infty} \leq \|u\|_{\infty}$. If we trace through the proof of Proposition \ref{propo: density argument} we can see that the following is also guaranteed:
	\begin{align*}
		\limsup_{k\rightarrow +\infty}\int_{J_{u_k}\cap \ov{A}}\varphi(x,\hat{u}^+_k,\hat{u}^-_k,\nu_k) \d\H^{n-1}(z)& \leq \int_{J_{\hat{u}}\cap \ov{A}} \varphi(x,\hat{u}^+,\hat{u}^-,\nu)\d\H^{n-1}(z),
	\end{align*}
for any $A\subseteq \hat{\Omega}$ and for any upper semicontinuous function $\varphi$ satisfying \eqref{palle1} and \eqref{palle2}. In particular, by testing the above inequality with $A=\Omega$,
    \[
    \varphi(x,\xi,\eta,\nu)=a\sqrt{\mathbb{A}(\xi-\eta)\odot \nu\cdot (\xi-\eta)\odot \nu}+F_{\infty}(z,(\xi-\eta)\odot\nu)
    \]
and with $\varphi=b$ we can infer that
	\begin{align*}
	    \limsup_{k\rightarrow +\infty}& \int_{J_{\hat{u}_k}\cap \ov{\Omega}}[a\sqrt{\mathbb{A}([\hat{u}_k]\odot \nu_k\cdot [\hat{u}_k]\odot \nu}+F_{\infty}(z,([\hat{u}_k]\odot\nu_k) +b]\d\H^{n-1}(z)\\
	    &\leq \int_{J_{\hat{u}}\cap\ov{\Omega}}[a\sqrt{\mathbb{A}([\hat{u}]\odot \nu\cdot [\hat{u}]\odot \nu}+F_{\infty}(z,([\hat{u}]\odot\nu) +b]\d\H^{n-1}(z).
	\end{align*}
Thus
	\begin{align*}
	\limsup_{k\rightarrow +\infty} &\F(\hat{u}_k,1;\ov{\Omega})\leq \F(u,1;\ov{\Omega})
	\end{align*}
By noticing that
		\[
		\F(\cdot ,1;\ov{\Omega})= \F(\cdot ,1;\Omega)+\mathcal{R}(\cdot, f), \ \ \ 
		\]
we conclude  by simply setting $u_k:=\hat{u}_k\ca_{\Omega}\in \Cl$. 
\end{proof}
We finally notice that the same diagonalization argument exploited in the proof of Theorem \ref{thm: limsup inequality} allows us to prove the following Proposition.
\begin{proposition}\label{propo final recovery}
Consider $\Omega$ to be an open bounded set with $C^1$ boundary and fix a boundary data $f\in C^1(\pa\Omega;\R^n)$. Let $\{\e_j\}_{j\in \N}$ be a vanishing sequence of real numbers. Then, for any $u\in SBD^2(\Omega)\cap L^{\infty}(\Omega)$ there exists a subsequence $\{\e_{j_k}\}_{k\in \N}\subset \{\e_{j}\}_{j\in \N}$ and a sequence of function $(u_{k},v_{k})\in H^1(\Omega;\R^n) \times V_{\e_{j_k}}$ such that
	\[
	\e_{j_k} \rightarrow 0, \ \ \ \|u_k-u\|_{L^2}\rightarrow 0, \ \ \ \|v_k-1\|_{L^2}\rightarrow 0,
	\]
with $u_{\e}=f$, $v_{\e}=1$ on $\pa \Omega$ and 
	\[
	\limsup_{k\rightarrow +\infty} \F_{\e_{j_k} }(u_{k},v_{k}) \leq  \F(u,1)+\mathcal{R}(u,f).
	\]
Moreover $\|u_k\|_{{\infty}}\leq \|u\|_{{\infty}}$.
\end{proposition}
 \subsection{Proof of Theorem \ref{thm: existence}}
We are finally in the position to prove Theorem \ref{thm: existence}.
\begin{proof}[Proof of Theorem \ref{thm: existence}]
We divide the proof in three steps. \\
\text{}\\
\textbf{Step one:} \textit{existence for $\gamma_{\e}$}. Fix $\e>0$ and consider $(u_k,v_{k})$, a  minimizing sequence. Then
	\[
	\sup_{k\in \N} \left\{\int_{\Omega} |e(u_k )|^2 \d x+\|v_k\|_{H^1} \right\}<+\infty.
	\]
In particular, Korn's inequality\footnote{the arbitrary rigid displacement is here fixed by the prescription of the boundary condition.} combined with the $L^1$-compactness for sequences with uniformly bounded $H^1$-norm gives us that $u_k \rightarrow u\in H^1$, $v_k \rightarrow v \in V_{\e}$ in $L^1$ and $e(u_k) \rightharpoonup e(u)$ in $L^2$ with also $u=f$, $v=1$ on $\pa \Omega$. Moreover, because of assumption \eqref{1.c} on $F$ and due to the uniform $L^2$ bound on the symmetric part of the gradient $e(u_k)$ we have
	\begin{align*}
	\left| \int_{\Omega} [F(x,e(u_k),v_k) - F(x,e(u_k),v) ]\d x\right|\leq \int_{\Omega} \omega_F(v_k;v)|e(u_k)|\d x \rightarrow 0.
	\end{align*}
Furthermore, due to the weak convergence of $e(u_k)$ and to the strong convergence of $v_k$ (see for example \cite[Theorem 2.3.1]{buttazzo1989semicontinuity})
	\begin{align*}
	\liminf_{k\rightarrow +\infty}\int_{\Omega} v_k \mathbb{A} e(u_k) \cdot e(u_k) \d x \geq \int_{\Omega} v \mathbb{A} e(u) \cdot e(u) \d x
	\end{align*}
All this considered yields, together with the convexity of $F(x,\cdot,v)$,
	\begin{equation}
	\liminf_{k\rightarrow +\infty} \mathcal{F}_{\e}(u_k, v_{k})\geq  \int_{\Omega} v \mathbb{A} e(u) \cdot e(u) dx +\frac{1}{\e}\int_{\Omega} \psi(v) \d x+\int_{\Omega} F(x,e(u),v)\d x.\nonumber
	\end{equation}
	In particular, by the application of the direct method of calculus of variation we achieve existence for $\gamma_{\e}$, $\e>0$.\\
	\text{}\\
\textbf{Step two:} \textit{liminf inequality}. Let $\{(u_{\e},v_{\e})\}_{\e>0} \subset H^1(\Omega;\R^n)\times V_{\e}$ be such that $u_{\e}\rightarrow u_0$, $v_{\e}\rightarrow 1$ in $L^1$ and with $u_{\e}=f$, $v_{\e}=1$ on $\pa \Omega$ for all $\e>0$. Then
	\begin{equation}\label{eqn:liminf bd datum}
	\liminf_{\e\rightarrow 0} \F_{\e}(u_{\e},v_{\e})\geq \F(u_0,1)+\mathcal{R}(u_0,f).
	\end{equation}
Indeed, by considering $\hat{\Omega}\supset \Omega$, a function $w\in H^1(\hat{\Omega};\R^n)$ with $w\Big{|}_{\pa \Omega}=f$ and the extension \	\[
\hat{u}_{\e}=u_{\e}\ca_{\Omega}+w\ca_{\hat{\Omega}\setminus \Omega}, \ \ \ \hat{v}_{\e}=v_{\e}\ca_{\Omega}+\ca_{\hat{\Omega}\setminus \Omega}
\]
we can notice that $\hat{u}_{\e}\rightarrow \hat{u}_0$, $\hat{v}_{\e}\rightarrow 1$. Moreover
	\[
	\F_{\e}(\hat{u}_{\e},\hat{v}_{\e};\hat{\Omega})=\F_{\e}(u_{\e},v_{\e};\Omega) + \int_{\hat{\Omega}\setminus \Omega} [\mathbb{A}e(w)\cdot e(w) + F(x,e(w),1)]\d x
	\]
and thanks to Theorem \ref{thm: liminf inequality} we have
\begin{align*}
\int_{\hat{\Omega}\setminus \Omega}& [\mathbb{A}e(w)\cdot e(w) + F(x,e(w),1)]\d x+\liminf_{\e\rightarrow 0} \F_{\e}(u_{\e},v_{\e};\Omega) = \liminf_{\e\rightarrow 0} \F_{\e}(\hat{u}_{\e},\hat{v}_{\e};\hat{\Omega}) \\
& \geq \F(\hat{u}_0,1;\hat{\Omega})=\F(u_0,1;\Omega)+\mathcal{R}(u_0,f) + \int_{\hat{\Omega}\setminus \Omega} [\mathbb{A}e(w)\cdot e(w) + F(x,e(w),1)]\d x,
\end{align*}
leading to \eqref{eqn:liminf bd datum}.\\
\text{}\\
\textbf{Step three:} \textit{proof of \eqref{eqn: limit min} and existence of a minimizer}. Let $\{\e_{j}\}_{j\in \N}$ be the sequence such that
$	\limsup_{\e\rightarrow 0} \gamma_{\e}=\lim_{j\rightarrow +\infty} \gamma_{\e_j}$. 
Thanks to Proposition \ref{propo final recovery} we have that, for any fixed $u_0\in SBD^2(\Omega)$ we can find $\{\e_{j_k}\}_{k\in \N}\subset \{\e_j\}_{j\in \N}$ and $(u_k,v_k)\in H^1(
\Omega;\R^n)\times V_{\e_{j_k}}$ with $u_k=f$ , $v_k=1$ on $\pa \Omega$ and such that it holds
	\[
	\F(u_0,1)+\mathcal{R}(u_0,f) \geq  \limsup_{k \rightarrow +\infty} \F_{\e_{j_k}}(u_{k},v_{k}) \geq  \limsup_{k\rightarrow +\infty} \gamma_{\e_{j_k}}=\limsup_{\e\to0} \gamma_{\e}.
	\]
Thus, by taking the infimum among $u_0\in SBD^2(\Omega)$ we get
	\begin{equation}\label{eqn:finita la esist}
	\gamma_{0}\geq \limsup_{\e \rightarrow 0} \gamma_{\e}.
	\end{equation}
On the other side, by denoting with $(\bar{u}_{\e},\bar{v}_{\e})$ the minimizers at the level $\gamma_{\e}$, we can ensure (thanks to the compactness Theorem \ref{main thm: comp}) that, there exists at least an accumuluation point and that any accumulation point has the form $(u_0,1)$ for some $u_0\in SBD^2(\Omega)$. Thus step two guarantees that
	\[
	\liminf_{\e\rightarrow 0} \gamma_{\e} = \liminf_{\e\rightarrow 0} \F_{\e}(\bar{u}_{\e},\bar{v}_{\e}) \geq \F(u_0,1)+ \mathcal{R}(u_0,f) \geq \gamma_0.
	\]
Combining this previous relation with \eqref{eqn:finita la esist} proves \eqref{eqn: limit min} and demonstrates also that any accumulation point of $\{(\bar{u}_{\e},\bar{v}_{\e})\}_{\e>0}$ provides a minimizer for $\gamma_0$.
\end{proof}

\section{Selected applications} \label{sct: Example}
We now provide examples of energy with some specific functions $F$ of interests with a view to applications.
As a title of example we consider the case where $\psi(v)=(1-v)^2$ yielding
$	a=2\sqrt{\alpha}$ and $b=\frac{2}{3}$.
\subsection{A simple model of fracking}
In the case of hydraulic fracturing, with a simple variational model as studied in \cite{xavier2017topological}, the phenomena is modeled through a potential of the type
	\[
	F(x,M,v)= - p(x,M,v) \trace(M).
	\]
We directly state the hypothesis on $p$ that guarantees  our $\Gamma$-convergence result \ref{thm: main thm g conv} and the existence Theorem \ref{thm: existence}. In particular, in order to apply our  results we require that the pressure $p$ is a concave function of the variable $M$ and that
	\begin{itemize}
	\item[1)] $p(\cdot,M,0)\in C^0(\Omega)$ for all $M\in M^{n\times n}_{sym}$;
	\item[2)] $p(x,\cdot,v)$ is a concave function for all $(x,v)\in \Omega\times [0,1]$; 
	\item[3)] $ -\sigma|x-y|\leq  p(x,M,v)-p(y,M,v)\leq \ell|x-y|$ for all $x,y\in \R^n$ and all $(M,v)\in  M_{sym}^{n\times n}\times [0,1]$
where $\ell>0$ is any real constant and 
	\begin{equation}\label{bound on p2}
	0<\sigma< \max_{\lambda\in (0,1)}\left\{ \frac{2\sqrt{\alpha \psi(\lambda)}}{\sqrt{\kappa}(1+2\sqrt{\alpha |\Omega|\psi(\lambda)/\lambda})}\right\}   < 2\sqrt{\frac{\alpha \psi(0)}{ \kappa }};
	\end{equation}
\item[4)] having set 
\begin{align*}
\omega_p(s;t)&:=\sup\left\{|p(x,M,s)-p(x,M,t)| \ : \ (x,M)\in \R^n\times M_{sym}^{n\times n}\right\}
\end{align*}
then 
	\[
	\displaystyle\lim_{s\rightarrow t} \omega_p(s;t)=0.
	\]
	\end{itemize}
Under these assumptions the potential $F=-p(x,M,v)\trace(M)$ satisfies Assumption \ref{hyp on F} and \eqref{1.a},\eqref{1.b},\eqref{1.c}. Moreover
	\[
	F_{\infty}(x,M)= -\trace(M) \lim_{t\rightarrow + \infty} p(x,tM,0).
	\]

\subsubsection{Pressure constant in $e(u)$ and linear in $v$}
We first examine the case 
	\[
	p(x,M,v):=(mv+q) \rho(x)
	\]
where $\rho\in L^{\infty}$ is a Lipschitz function and $m,q\in\R$. Provided $\rho$ has suitably small $L^{\infty}$ norm, hypothesis 1), 2), 3) and 4)  are clearly satisfied. We have
	\[
	p(x,M,1)=(m+q)\rho(x), \ \ p(x,M,0)=q\rho(x).
	\]
Moreover $F_{\infty}(x,M)=q\rho(x)\trace(M)$.
Hence the $\Gamma-$limit of the energy \eqref{eqn:energy e} is given by
	\begin{align*}
	\Phi(u)&:=\int_{\Omega} [\mathbb{A} e(u)\cdot e(u) - (m+q)\rho(x) \dive(u)]\d x +b\H^{n-1}(J_u)+\\
	&+a \int_{J_u} \sqrt{\mathbb{A} ([u]\odot \nu \cdot) ([u]\odot \nu)} \d \H^{n-1}(z)- q\int_{J_u} \rho(z) [u](z)\cdot \nu(z) \d \H^{n-1}(z).
	\end{align*}
		The model in \cite{xavier2017topological} corresponds to $m=0$ and $\rho$ is a constant taken as a hydrostatic pressure acting as a boundary condition inside the crack considered as impermeable. Note that in \cite{xavier2017topological} exactly the approximation of this work is proposed. Another phase-field approximation closer to the original Ambrosio-Tortorelli model is considered in \cite{B2013}, with $q=0$ and a constant $\rho$. Note however that their claimed limit functional is not what we proved to be.
\subsubsection{Pressure non constant in $e(u)$: isotropic and anisotropic case}
We now examine the case where the pressure $p$ has a concave dependence on the variable $M$:
	\[
	p(x,M,v):=\rho(x,v)g(M).
	\]
A suitable choice of $\rho$ ensures that 1) and 2) are in force. In order to guarantee 3) (and thus 4) provided a suitable $\rho$) we ask also that $\|g\|_{ L^{\infty}}<c$ for an appropriate constant $c$. In particular any concave bounded function is such that
	\[
	\lim_{t\rightarrow +\infty} g(tM)=\gamma(M)
	\]
exists finite. Thus the $\Gamma-$limit of the energy \eqref{eqn:energy e} is given by
\begin{align*}
	\Phi(u):=&	\int_{\Omega} [\A e(u)\cdot e(u) -  \rho(x,1) g(e(u))\dive(u)] \d x\\
	&+a \int_{J_u} \sqrt{\A ([u](z)\odot \nu(z)) \cdot ([u](z)\odot \nu(z)) } \d \H^{n-1}(z)\\
&+b \H^{n-1}(J_u)- \int_{J_u} \rho(z,0) \gamma([u](z)\odot \nu(z)) [u](z) \cdot \nu(z) \d \H^{n-1}(z).
	\end{align*}
	This case corresponds to a more realistic  fracking model where the pressure is a thermodynamic variable with a certain constitutive law (as related to the Biot's coefficient and the pore-pressure \cite{B2013}), instead of a hydrostatic pressure given as a model datum. In particular this case applies to the
case where the crack is no more impermeable, and hence the pressure satisfies a certain balance equation in the whole domain.

\subsection{Pressure almost constant in $x$: the two-rocks model }

Of particular interest in the case of hydraulic fracking is the case where the pressure $p$ takes values $p_1(e(u)), p_2(e(u))$ in two different region of our ambient space $\Omega$ and quickly varies from $p_1$ to $p_2$ in a small layer of size $\delta$ bordering the two regions. This models the situation of a so-called stratified domain, i.e., where we have two 
permeable rocks (or impermeable if $p_i$ assumes a constant value in  each rock) separated by an interface (where the pressure is linearly interpolated). As a title of example we consider the situation depicted in figure \ref{fig: one}. In particular we set
\begin{equation}
p(x,M,v):=\left\{
\begin{array}{ll}
\rho(v)p_1(M) & \text{if $x\in Q_1$ and $d(x,S)>\delta$};\\
\rho(v) p_2(M) & \text{if $x\in Q_2$ and $d(x,S)>\delta$};\\
 \frac{d(x,S)}{\delta}  \rho(v) p_1(M)  & \text{if $x\in Q_1$ and $d(x,S)\leq \delta$};\\
\frac{d(x,S)}{\delta} \rho(v) p_2(M)  & \text{if $x\in Q_2$ and $d(x,S)\leq \delta$};
\end{array}
 \right.
\end{equation}
If $p_i$ are concave function and $\|\rho p_i\|_{\infty}$ is suitably small, we can surely choose $\rho$ so that conditions 3) and 4) are satisfied. Moreover, setting 
	\[
	p_i^{\infty}(x,M)=\lim_{t\rightarrow +\infty} p_i(tM)
	\]
and
\begin{equation}
p_{\infty}(x,M):=\left\{
\begin{array}{ll}
\rho(0)p_1^{\infty}(M) & \text{if $x\in Q_1$ and $d(x,S)>\delta$};\\
\rho(v) p_2^{\infty}(M) & \text{if $x\in Q_2$ and $d(x,S)>\delta$};\\
 \frac{d(x,S)}{\delta}  \rho(0) p_1^{\infty}(M)  & \text{if $x\in Q_1$ and $d(x,S)\leq \delta$};\\
\frac{d(x,S)}{\delta} \rho(0) p_2^{\infty}(M)  & \text{if $x\in Q_2$ and $d(x,S)\leq \delta$};
\end{array}
 \right.
\end{equation}
we get that the limiting energy reads
\begin{align*}
	\Phi(u):=&	\int_{Q} [\A e(u)\cdot e(u) -  p(x,e(u),1)\dive(u)] \d x\\
	&+a \int_{J_u} \sqrt{\A ([u](z)\odot \nu(z)) \cdot ([u](z)\odot \nu(z)) } \d \H^{n-1}(z)\\
&+b \H^{n-1}(J_u)-   \int_{J_u}   p_{\infty}(x,[u](z)\odot \nu(z)) [u](z) \cdot \nu(z) \d \H^{n-1}(z),
	\end{align*}
	that can be rearranged as
	\begin{align*}
	\Phi(u):=&	\sum_{i=1}^2\int_{Q_i\setminus Q_{\delta} } [\A e(u)\cdot e(u) -  \rho(1) p_i(e(u))\dive(u)] \d x\\
	&-   \rho(0)\int_{J_u\cap (Q_i\setminus Q_{\delta} )}  p_{i}^{\infty}([u](z)\odot \nu(z)) [u](z) \cdot \nu(z) \d \H^{n-1}(z)\\
	& +	\sum_{i=1}^2\int_{Q_i\cap Q_{\delta} } [\A e(u)\cdot e(u) -  \frac{d(x,S)\rho(1)}{\delta} p_i(e(u))\dive(u)] \d x\\
	&-   \rho(0)\int_{J_u\cap (Q_i\setminus Q_{\delta} )}   \frac{d(x,S)}{\delta} p_{i}^{\infty}([u](z)\odot \nu(z)) [u](z) \cdot \nu(z) \d \H^{n-1}(z)\\
	&+a \int_{J_u} \sqrt{\A ([u](z)\odot \nu(z)) \cdot ([u](z)\odot \nu(z)) } \d \H^{n-1}(z)+b \H^{n-1}(J_u).
	\end{align*}
The case with several rocks can be obtained in the same way.
\begin{figure}\label{fig: one}
\includegraphics[scale=0.8]{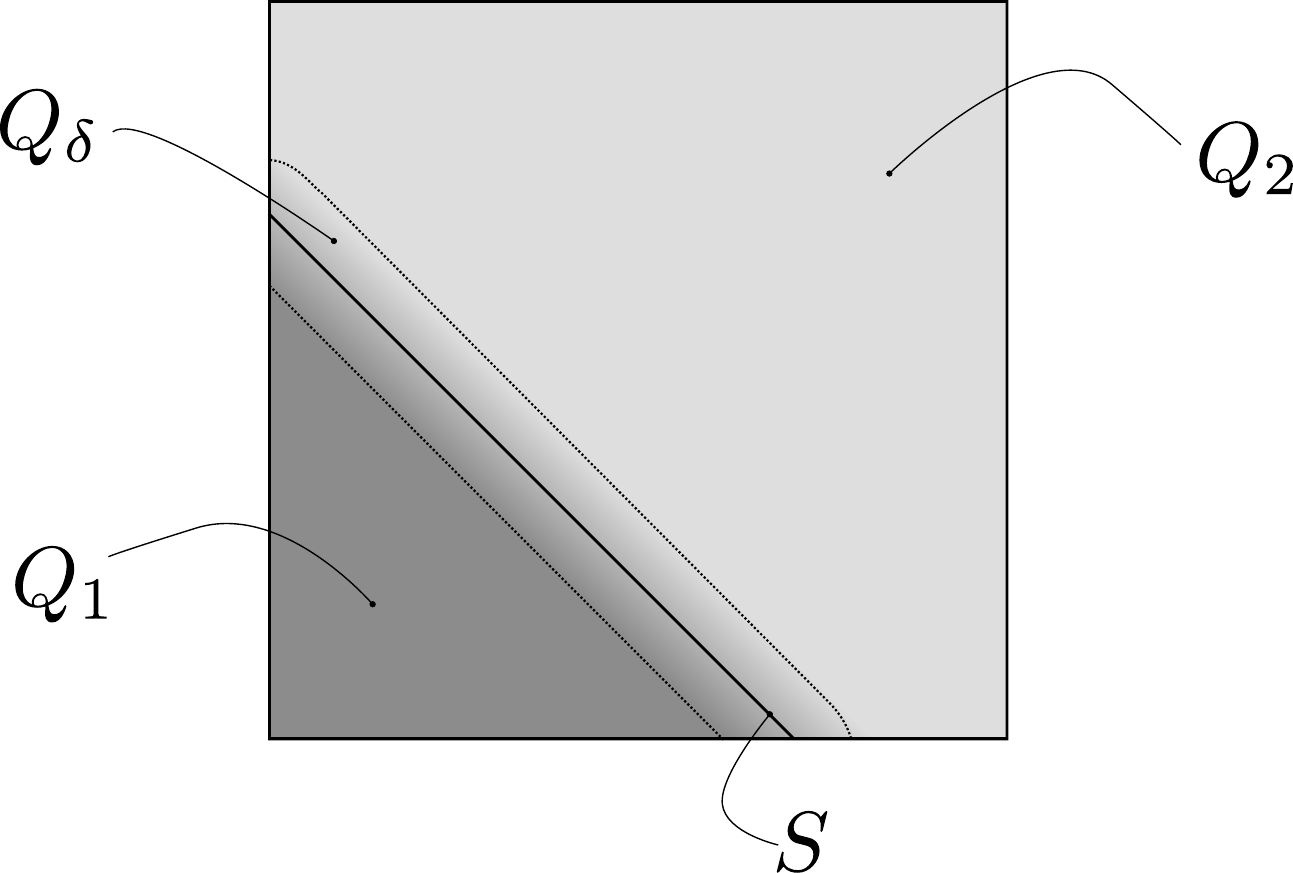}\caption{The two different materials behave differently when subject to an elastic strain. This is modeled by considering two different pressures on each component. In the picture, different gray corresponds to different value of $p(\cdot,M,v)$. Notice that the role of the layer ($\delta$) around the interface $S$ can be made as small as we like and it is adopted only to satisfy the continuity assumption on the spatial behavior of the pressure and to take into account eventual situations where $\H^{n-1}(J_u\cap \pa Q_1 \cap \pa Q_2 )>0$.}
\end{figure}
\subsection{A model of plastic slip: $F=p|e(u)|$} Now we analyze the case 
	\[
	F(x,M,v):= p(x,v) g(|M|),
	\]
	that consists of a generalization of a phase-field approximation of plastic slip as discussed in \cite{AmbrosioElast} for the anti-plane case.

By possibly making additional restriction on the function $p$, a functional dependence on $M$ can be considered. However, for the sake of clarity and as a title of example we would avoid such a dependence. It is immediate that
	\[
	F_{\infty}(x,M)=p(x,0)g_{\infty} |M|,
	\]
where $g_{\infty}:=\lim_{t\rightarrow +\infty} \frac{g(t)}{t}$. Thus, the limit energy in this scenario is
	\begin{align*}
	\Phi(u):=&	\int_{\Omega} \left[\A e(u)\cdot e(u) -  p(x,1)g(|e(u)|)\right]  \d x\\
	&+a \int_{J_u} \sqrt{\A ([u](z)\odot \nu(z)) \cdot ([u](z)\odot \nu(z)) } \d \H^{n-1}(z)\\
&+b \H^{n-1}(J_u)+g_{\infty} \int_{J_u} p(z,0)  |[u](z)\odot \nu(z)|  \d \H^{n-1}(z).
	\end{align*}

\subsection{The Tresca yield model in elasto-plasticity: $F=\lambda_{max}(\mathbb{A}e(u)) - \lambda_{min}(\mathbb{A}e(u))$}This is so far an academic example in the sense that no such criterion, though important in engineering, is known to the authors as implemented in any variational setting so far.
Nevertheless, interpreting $p$ as a Lagrange multiplier, could provide a model with a sort of averaged Tresca threshold. Consider the operators
	\[
	\lambda_{max}(\mathbb{A}M):=\max_{i=1,\ldots,n}\{\lambda_i(\mathbb{A}M)\}
	\]
and
	\[
	\lambda_{min}(\mathbb{A}M):=\min_{i=1,\ldots,n}\{\lambda_i(\mathbb{A}M)\}
	\]
where $\lambda_i(P)$ denotes the $i$-th eigenvalue of the matrix $P$. This function are, respectively convex and concave and  
	\[
	\lambda_{max}(\mathbb{A}e(u))-\lambda_{min}(\mathbb{A}e(u))\leq \lambda_{max}(\mathbb{A}e(u))\leq \|\mathbb{A}\| |M|.
	\]
Hence, by setting
	\[
	F(x,M,v)=p(x,v)g( \lambda_{max}(\mathbb{A}e(u))-\lambda_{min}(\mathbb{A}e(u)) ) 
	\]
%
provided $g$ is a convex function with sublinear growth, the class of function $p$ such that hypothesis \ref{hyp on F} and \eqref{1.a},\eqref{1.b} and \eqref{1.c} on $F$ are satisfied is not trivial.
Notice now that
	\begin{align*}
	\lambda_{max}(t\mathbb{A}M)=t\lambda_{max}(\mathbb{A}M), \ \ \ \ \ \lambda_{min}(t\mathbb{A}M)=t\lambda_{min}(\mathbb{A}M), 
	\end{align*}
and thus, as above, we get
	\[
	F_{\infty}(x,M)=g_{\infty}p(x,0) (\lambda_{max}(\mathbb{A}M) - \lambda_{min}(\mathbb{A}M)).
	\]
where $	g_{\infty}=\lim_{t\rightarrow +\infty} \frac{g(t)}{t}$. The limit energy here is 
	\begin{align*}
	\Phi(u):=&	\int_{\Omega} \left[\A e(u)\cdot e(u) -  p(x,1)g(\lambda(\mathbb{A}e(u) ))\right]  \d x\\
	&+a \int_{J_u} \sqrt{\A ([u](z)\odot \nu(z)) \cdot ([u](z)\odot \nu(z)) } \d \H^{n-1}(z)\\
&+b \H^{n-1}(J_u)+g_{\infty}\int_{J_u} p(z,0)(\lambda_{max}(\mathbb{A}([u]\odot \nu) )-\lambda_{min}(\mathbb{A} ([u]\odot \nu) ))  \d \H^{n-1}(z).
	\end{align*}
	
	\subsection{The non-interpenetration condition}
It is well-known that an opening crack should satisfy the non-interpenetration condition $[u]\cdot \nu\geq0$ which is not enforced so far by the Lagrangians we considered. In particular we would like to have a model where $([u]\cdot \nu)^-$ is not energetically influent in the evolution of the system.  Having set $\mathcal{H}(u):=\int_{J_u}|([u]\cdot \nu)^-|(z)d\H^{n-1}$, from a variational point of view we can define a minimization problem for an energy $\mathcal{G}$ subject to a non-interpenetration condition as
	\begin{align}\label{NI}
	\inf\{\mathcal G(u)|  \ u\in\mathcal A_d \mbox{ and } \H(u)=0\},
	\end{align}
where $\mathcal{A}_d$ is a suitable admissible class. The associated Lagrangian to such a problem reads as
	$$
	\mathcal L(u,p):=\mathcal G(u)+\int_{J_u}p(z)([u]\cdot \nu)^-(z)d\H^{n-1},\quad (u,p)\in\mathcal A_d\times L^\infty(J_u;\H^{n-1}).
	$$
	It is a well-known result of convex optimization (see e.g. \cite[Proposition 3.3.]{bonnans2013perturbation}) that if
	\begin{align}\label{opt}
	\displaystyle (u,p)\in \displaystyle\mathrm{arg}\min_{(u,p)\in \mathcal A_d} \mathcal \{\L(u,p)\}, \ \ \ \  \mbox{ with } p\mbox{ s.t. }\int_{J_u}p(z)([u]\cdot \nu)^-(z)d\H^{n-1}\geq0,
	\end{align}
	then $u$ is a solution of \eqref{NI}. Following our approach we can write a Lagrangian by exploiting our lower order potential $F$. An appropriate low-order potential for problem \eqref{opt} can be chosen as
	$$
	F(x,M,v)=(1-v)^2p(x)\max\{-\trace(M), 0\}=(1-v)^2 p(x) \trace(M)^- .
	$$
Notice that, $M\mapsto \max\{-\trace(M), 0\}$ is a positive convex function and with sublinear growth (since $|\max\{a,b\}|\leq |a|+|b|).$ In particular a suitable choice of $p$ will ensure that our hypothesis on $F$ \ref{hyp on F} together with \eqref{1.a},\eqref{1.b} and \eqref{1.c} are satisfied. Notice that, for $t>0$, one has $\trace(tM)^-=t\max\{-\trace(M), 0\}$, 
and thus
	\[
	F_{\infty}(x,M)=p(x)\trace(M)^- 
	\]
With these carefulness we can $\Gamma$-approximate the Lagrangian
	\begin{align*}
	\Phi(u):=&	\int_{\Omega}  \A e(u)\cdot e(u)\d x+a \int_{J_u} \sqrt{\A ([u](z)\odot \nu(z)) \cdot ([u](z)\odot \nu(z)) } \d \H^{n-1}(z)\\
&+b \H^{n-1}(J_u)+\int_{J_u} p(z)([u]\cdot \nu)^- \d \H^{n-1}(z)
	\end{align*}
by
	\begin{align*}
	\F_{\e}(u,v)=\int_{\Omega}  v \A e(u)\cdot e(u) +\frac{1}{\e}\int_{\Omega} \psi(v) d x  +\int_{\Omega} (1-v)^2 p(x) \dive(u)^{-}\d x.
	\end{align*}
\section{Appendix} \label{sct: appendix}

\subsection{A semicontuity result on $SBD$}
We now proceed to the proof of a lower semicontinuity result. This result can be derived by gathering several results available in the  literature. We retrieve them here and we give a brief sketch of the proof of the main result in order to present our work as self-contained as possible . Let us start with the following Proposition:

\begin{proposition}\label{propo: justification of limit}
For any fixed $L\in M_{sym}^{n\times n}$ there exists a function $F_{\infty}(x,M;L):M_{sym}^{n\times n} \rightarrow \R$ such that
	\[
	\lim_{t\rightarrow+\infty} \frac{F(x,L+tM,0) - F(x,L,0)}{t} =F_{\infty}(x,M;L).
	\]
Moreover $F_{\infty}(x,r M;L)=rF_{\infty}(x,M;L)$ for all $r\in \R^+$.
\end{proposition}
\begin{proof}
Consider, for fixed $M\in \times M_{sym}^{n\times n}$ and $x\in \Omega$ the quantity
		\[
		f(t):=\frac{F(x,L+tM,0)-F(x,L,0)}{t}.
		\]
Due to the convexity of $F(x,\cdot, 0)$ we deduce that $f(t)$ is increasing on $(0,+\infty)$. Moreover, assumption 1) in \ref{hyp on F}  also guarantees that
	\[
	\left| \frac{F(x,L+tM,0)-F(x,L,0)}{t}\right| \leq \ell \left(\frac{|L| - F(x,L,0)}{t} \right) + \ell |M|.
	\]
In particular 
	\[
	\lim_{t\rightarrow+\infty} f(t)=\text{exists finite}.
	\] 
Thus there exists a function $F_{\infty}(x,M;L)$ such that
	\[
	\lim_{t\rightarrow+\infty} f(t)=F_{\infty}(x,M;L).
	\] 
By definition of $F_{\infty}$ we have finally that
	\begin{align*}
	F_{\infty}(x,rM,0;L)&=rF_{\infty}(x,M;L).
	\end{align*}
	\end{proof}
\begin{remark}\rm{We can think the function $F_{\infty}(x,\cdot ;L)$ as a function defined on the unit sphere of $M_{sym}^{n\times n}$ and extended homogeneusly on the whole space. }
\end{remark}

The first thing we need is the following  decomposition Lemma, holding for convex function with suitable regularity, which as a Corollary yields the independence of the function $F_{\infty}$ from the starting point $L$.

\begin{proposition}\label{prop: convex decomposition}
Let $G:\Omega \times M_{sym}^{n\times n} \rightarrow \R$ be a function such that $G(x,M)$ is lower semicontinuous in $(x,M)$, $G(x,\cdot)$ is convex for all $x\in \Omega$ and $|G(x,M)|<\ell |M|$ for some $\ell\in \R$ and for all $(x,M)$. Then there exists two families of continuous function $\{ a_j(x): \Omega \rightarrow M_{sym}^{n\times n} \}_{j\in\N}$ and $\{b_j(x): \Omega \rightarrow \R\}_{j\in \N}$ such that
	\[
	G(x,M)=\sup_{j\in \N}\{ a_j(x)\cdot M + b_j(x)\}
	\]
and 
	\[ 
	\lim_{t\rightarrow +\infty} \frac{G(x,L+tM)-G(x,L)}{t} =\sup_{j\in \N}\{ a_j(x)\cdot M \}
	\]
 for any $L\in M_{sym}^{n\times n}$.
\end{proposition}
The proof of the previous Proposition comes as a consequence of \cite[Lemma 2.2.3, Remark 2.2.6, Lemma 3.1.3]{buttazzo1989semicontinuity}.
\begin{remark}
Let us briefly sketch the proof of Proposition \ref{prop: convex decomposition} in the easy case where $G(x,M)=G(M)$ is convex just to give an idea to the reader about why such decomposition hold true (the proof can be also found in \cite{AFP}.  Chosen $\{P_j\}_{j\in \N} \subset M_{sym}^{n\times n}$ a dense set it is enough to define the values
	\[
	a_j:=\nabla_{M} G(P_j), \  \ \ \ b_j:=-  \nabla_{M}  G(P_j) \cdot P_j + G(P_j).
	\]
Notice that
\begin{align}
	G(P_{j}) = \nabla_M G(P_j) \cdot P_j  - \nabla_M G( P_j) \cdot P_j + G(P_j)=a_j \cdot P_j  + b_j.\label{gigi}
	\end{align}
Pick now any $M\in M_{sym}^{n\times n}$ and let $\{P_{j_k}\}_{k\in \N} \subset\{P_j\}_{j\in \N}$ be a subsequence such that $P_{j_k}\rightarrow M$. Since $G(x,\cdot)$ is convex and thanks to \eqref{gigi} we get
$G( M) \geq  a_j \cdot M  + b_j$, 
and hence
	\[
	G(M) \geq \sup_{j\in \N} \{  a_j \cdot M + b_j \}.
	\]
On the other hand, by continuity, $G( M)=\lim_k G( P_{j_k})$ and thus for any $\delta>0$ there exists $K_0$ such that
	\[
	G( M)\leq G (P_{j_k})+\delta \ \ \ \ \text{for all $k\geq K_0$}.
	\]
Thus
	\begin{align}
	G( M)\leq &G(P_{j_k})+\delta = a_{j_k} \cdot P_{j_k}  + b_{j_k}+\delta\nonumber\\
	=& a_{j_k}\cdot M  + b_{j_k}(x)+\delta+ a_{j_k}\cdot (P_{j_k}-M) \nonumber\\
		\leq& \sup_{j\in \N} \{  a_{j} \cdot M + b_{j} \}+\delta+  a_{j_k} \cdot (P_{j_k}-M).\rangle\label{proietti}	
		\end{align}
Function $G$ being convex it is also Liptshitz on  every bounded set in $M_{sym}^{n\times n}$ and in particular $a_{j_k}=\nabla_M G(P_{j_k})$ is bounded for $P_{j_k}$ close enough to $M$. Thus $a_{j_k} \cdot (P_{j_k}-M) \rightarrow 0$ and in particular, by taking the limit in $k$ and then in $\delta$ in \eqref{proietti}, we get
	\[
	G(M) \leq \sup_{j\in \N} \{  a_j \cdot M + b_j \}.
	\]
For the recession function instead we see that,  because of the convexity, for any $L\in M_{sym}^{n\times n}$ the quantity
	\[
	\frac{G(L+tM) - G(L)}{t}
	\]
is increasing in $t$ and thus
	\begin{align*}
	\lim_{k \rightarrow +\infty} \frac{G(L+kM) - G(L)}{k} &=\sup_{k\in \N} \frac{G(L+kM)-G(L)}{k}.
	\end{align*}
On the one hand, for all $j\in \N$, we get
	\[
\frac{G(L+kM)-G(L)}{k} \geq \frac{a_j\cdot L+b_j -G(L)}{k} + a_j \cdot M ,
	\]
which implies	
	\[
\lim_{k\rightarrow +\infty} \frac{G(L+kM) - G(L)}{k} \geq \sup_{j\in \N} \left\{a_j\cdot M \right\}.
	\]
On the other, for any $k\in \N$,  it holds
	\begin{align*}
	\frac{G(L+kM) - G(L)}{k}&=\sup_{j\in \N} \left\{\frac{b_j + a_j \cdot L -G(L)}{k} + a_j\cdot M\right\}\\
	&\leq \sup_{j\in \N} \left\{a_j \cdot M  \right\},
	\end{align*}
since $b_j+a_j\cdot L \leq G(L)$. In particular the equality is attained.
\end{remark}
\begin{remark}
In the light of Proposition \ref{prop: convex decomposition} it is clear that the recession function is independent of the starting point $L$.
\end{remark}
We recall the following technical Lemma from \cite[Lemma 2.35]{AFP}.
\begin{lemma}\label{lem:sup}
Let $\nu$ be any positive Radon measure and let $\varphi_i:\Omega \rightarrow \R^+$ with $i\in \N$ be a family of Borel functions. Then
	\[
	\sup_{i\in \N} \left\{\sum_{i\in I } \int_{A_i} \varphi_i(x) \d \nu (x) \right\} = \int_{\Omega}	\sup_{i\in \N} \left\{\varphi_i(x) \d \nu (x) \right\}  \d \nu(x).
	\]
where the supremum ranges over all finite families $\{A_i\}_{i\in I}$ of pairwise disjoint open set compactly contained in $\Omega$.
\end{lemma}
We now state and prove the semicontinuity result. For the sake of completeness we mention that this result comes also as a consequence of \cite[Theorem 3.4.1, Corollary 3.4.2 ]{buttazzo1989semicontinuity}.
\begin{proposition}\label{prop:limit of convex}
Let $G:\Omega\times M_{sym}^{n\times n} \rightarrow \R_+$ be a positive function such that $G(x,M)$ is lower semicontinuous in $(x,M)$, $G(x,\cdot)$ is convex for all $x\in \Omega$ and $|G(x,M)|<\ell |M|$ for some $\ell\in \R$ and for all $(x,M)$. Then, for any $u_{\e}\in H^1(\Omega;\R^n)$ such that $u_{\e}\rightarrow u$ in $L^1$ with $u\in SBD^2(\Omega)$ it holds
	\begin{align*}
	\liminf_{\e\rightarrow 0 } \int_{A}G(x,e(u_{\e})) \d x \geq  \int_{A} G(x,e(u))\d x + \int_{A\cap J_u}  G_{\infty}(z,[u]\odot \nu) \d \H^{n-1}(z) 
	\end{align*}
for all open set $A\subset \Omega$.
\end{proposition}
\begin{proof}
We first notice that, since $u_{\e}\rightarrow u$ in $L^1$ and $u_{\e},u \in SBD^2(\Omega)$ we have 
	\[
	e(u_{\e})\L^n \rightharpoonup^* Eu.
	\]
Fix $A\subset \Omega$. We can apply Proposition \ref{prop: convex decomposition} to find two families of continuous functions $a_j(x):\Omega\rightarrow M_{sym}^{n\times n}$, $b_j(x):\Omega\rightarrow \R$ such that
	\[
	G(x,M)=\sup_{j\in \N}\{ a_j(x)\cdot M +b_j(x)\}, \ \ \ \ G_{\infty}(x,M)=\sup_{j\in \N}\{a_j(x) \cdot M\}.
	\]
Let $A_0,\ldots A_m$ be pairwise disjoint open subset of $A$ and $\varphi_j\in C_c(A_j)$ with $0\leq \varphi_j\leq 1$ for all $j=0,\ldots, m$. Then
	\begin{align*}
	\int_{A} G(x,e(u_{\e}))\d x &\geq \sum_{j=0}^m \int_{A_j} \varphi_j G(x,e(u_{\e}))\d x\\
	&\geq \sum_{j=0}^m \int_{A_j} \varphi_j a_j(x)\cdot e(u_{\e})  \d x + \int_{A_j} \varphi_j b_j(x) \d x,
	\end{align*}
which, by passing to the limit in $\e$ and by exploiting the fact that $a_j \varphi_j\in C_c(A_j;M_{sym}^{n\times n} )$ leads to
 \begin{align*}
	\liminf_{\e \rightarrow 0} \int_{A} &G(x,e(u_{\e})\d x \geq \sum_{j=0}^m \int_{A_j} \varphi_j a_j(x)\cdot d Eu(x) + \int_{A_j} \varphi_j b_j(x) \d x\\
	&= \sum_{j=0}^m \int_{A_j} \varphi_j \left[a_j(x)\cdot \frac{d Eu}{d \L^n} (x)+b_j(x) \right] \d x+ \int_{A_j} a_j(x)\cdot d Eu^s(x) \\
	&= \sum_{j=0}^m \int_{A_j} \varphi_j \left[a_j(x)\cdot e(u)+b_j(x) \right] \d x+ \int_{A_j\cap J_u} \varphi_j a_j(x)\cdot ([u]\odot \nu ) \d \H^{n-1}(x).
	\end{align*}
We now want to apply Lemma \ref{lem:sup} and thus we set $\nu=\L^n + \H^{n-1}\llcorner J_u$ and we define the functions
	\begin{equation}
	\begin{split}
	\phi_j(x)&:=\left\{
	\begin{array}{ll}
	a_j(x) \cdot e(u) +b_j(x) & \text{for $x\in A \setminus J_u$}\\
	a_j(x) \cdot ([u]\odot \nu) & \text{for $x\in J_u \cap A$},
	\end{array}
	\right.,\\
	& \text{}\\
	\phi(x)&:=\left\{
	\begin{array}{ll}
	G(x,e(u)) & \text{for $x\in A \setminus J_u$}\\
	G_{\infty}(x,[u]\odot \nu)  & \text{for $x\in J_u\cap A$}.
	\end{array}
	\right. 
		\end{split}
	\end{equation}
Notice that, due to the mutual singularity of $\L^n$ and $\H^{n-1}\llcorner J_u$, we get
	\begin{align*}
	\sum_{j=0}^m \int_{A_j} \phi_j \varphi_j \d \nu \leq \liminf_{\e\rightarrow 0} \int_{A} G(x,e(u_{\e})) \d x.
	\end{align*}
By taking the supremum over $\varphi_j$ we get
	\begin{align*}
	\sum_{j=0}^m \int_{A_j} \phi_j^+ \d \nu \leq \liminf_{\e\rightarrow 0} \int_{A} G(x,e(u_{\e}) ) \d x.
	\end{align*}
Thanks to Proposition \ref{prop: convex decomposition}, for any fixed $x\in A$ it holds
	\[
	\sup_{j\in \N} \{\phi_j(x)\}=	\sup_{j\in \N} \{\phi_j^+(x)\}=\phi(x),
	\]
since $\phi\geq0$ for all $x\in \Omega$. Now, by taking the supremum among all the finite families of pairwise disjoint open subsets of $A$ and by applying Lemma \ref{lem:sup}, we get 
	\begin{align*}
	\liminf_{\e\rightarrow 0} \int_{A} G(x,e(u_{\e})) \d x&\geq \sup_{i\in I} \left\{\sum_{i\in I} \int_{A_i} \phi_j^+\d \nu(x)\right\}=\int_A \sup_{j\in N} \{\phi_j^+(x)\}\d \nu(x)\\
&=\int_{A} G(x,e(u))\d x + \int_{A\cap J_u} G_{\infty}(z, [u]\odot \nu) \d \H^{n-1}(z).
	\end{align*}
\end{proof}
\subsection*{Acknowledgements}
The authors were supported by the FCT Starting Grant `` Mathematical theory of dislocations: geometry, analysis, and modelling'' (IF/00734/2013).

\bibliography{references}
\bibliographystyle{plain}

\end{document}